\documentclass[11pt,a4paper]{amsart}

\usepackage[utf8]{inputenc}
\usepackage{mathrsfs}
\usepackage{mathdots}

\usepackage[left=3cm, right=3cm]{geometry}
\usepackage{txfonts}
 
\usepackage{graphicx}
\usepackage{hyperref}
\usepackage{amssymb}
\usepackage{amsmath}
\usepackage{verbatim,enumerate}
\usepackage{tikz-cd}
\usepackage{relsize}

\usepackage{cancel}

\numberwithin{equation}{section}
\numberwithin{figure}{section}

\newtheorem{thm}{Theorem}[section]
\newtheorem{lemma}[thm]{Lemma}
\newtheorem{prop}[thm]{Proposition}
\newtheorem{cor}[thm]{Corollary}

\theoremstyle{definition}
\newtheorem{definition}[thm]{Definition}
\newtheorem{example}[thm]{Example}
\newtheorem{rem}[thm]{Remark}

\newtheorem*{theorem_intro}{Theorem}

\newcommand{\C}{\mathbb{C}}

\newcommand{\M}{\mathcal{M}}
\newcommand{\U}{\mathcal{U}}
\newcommand{\N}{\mathcal{N}}
\newcommand{\E}{\mathcal{E}}

\newcommand{\CP}{\mathbb{C}\mathbf{P}}

\renewcommand{\P}{\mathbb{P}}

\renewcommand{\epsilon}{\varepsilon}
\newcommand{\R}{\mathbb{R}}

\newcommand{\Z}{\mathbb{Z}}

\newcommand{\sC}{\mathscr{C}}

\newcommand{\Ga}{\Gamma}

\newcommand{\cF}{\mathcal F}

\renewcommand{\geq}{\geqslant}

\renewcommand{\leq}{\leqslant}

\renewcommand{\phi}{\varphi}

\newcommand{\Aut}{\operatorname{Aut}}

\newcommand{\id}{\mathrm{id}}
\newcommand{\Mor}{\mathrm{Mor}}
\newcommand{\Div}{\mathrm{Div}}

\newcommand{\An}{\mathcal{A}n}
\newcommand{\X}{\mathcal{X}}
\newcommand{\Y}{\mathcal{Y}}

\newcommand{\ul}[1]{\underline{#1}}

\newcommand{\Gal}{\mathrm{Gal}}

\newcommand{\SL}{\mathbf{SL}}
\newcommand{\PSL}{\mathbf{PSL}}
\newcommand{\GL}{\mathbf{GL}}

\newcommand{\fh}{\mathfrak{h}}

\newcommand{\lmt}{\longmapsto}
\newcommand{\lra}{\longrightarrow}

\def\Circlearrowright{\ensuremath{%
  \rotatebox[origin=c]{180}{$\circlearrowright$}}}

\title{Orbifolds and the modular curve}

\author{Juan Martín Pérez}
\address{Freie Universität Berlin, Arnimallee 3, Raum 011, 14195 Berlin, Germany.}
\email{martin.perez@fu-berlin.de}

\author{Florent Schaffhauser}
\address{Mathematisches Institut, Universität Heidelberg, Im Neuenheimer Feld 205, 69120 Heidelberg, Germany}
\email{fschaffhauser@mathi.uni-heidelberg.de}

\begin{document}

\begin{abstract}
We provide an account of the construction of the moduli stack of elliptic curves as an analytic orbifold. While intimately linked to Thurston's point of view on the subject (discrete groups acting properly and effectively on differentiable manifolds), the construction of the modular orbi-curve and its universal family of elliptic curves ends up requiring a bit more technology, in order to allow for non-effective actions.
\end{abstract}

\date{\today}

\subjclass{Primary 14D23; Secondary 32G13}
\keywords{Stacks and moduli problems; Complex-analytic moduli problems}

\maketitle

\vspace{-15pt}

\tableofcontents

\vspace{-25pt}

\section*{Introduction}

Orbifolds were first introduced by Satake, under the name $V$-manifolds (\cite{Satake_generalized_manifolds}), as a way to handle the \textit{ramification} phenomena that occur in the general theory of group actions on manifolds. And it is generally accepted in the community that Satake chose that name in reference to the German word \textit{Verzweigung}, which means ramification.\footnote{The second-named author believes he heard this explanation of the term in a talk by Professor Kaoru Ono, who asked Satake himself directly about it.} Now, ramification is a concept that is also present in the analytic or algebraic contexts, in particular to talk about morphisms that are \textit{not étale}. And it is in fact remarkable that some of the concepts introduced by Satake and Thurston for differential-geometric purposes, namely orbifolds and their fundamental groups (\cite{Satake_Gauss_Bonnet,Thurston_2002}), have natural algebro-geometric counterparts that can in turn be used to enrich our perspective and the scope of our methods in differential and analytic geometry. One of the goals of the present chapter is to illustrate this point of view through the example of the moduli space of elliptic curves.

\smallskip

It is well-known that elliptic curves can be classified by their $j$-invariant and this is the basis for the construction of the so-called modular curve, whose points parameterise isomorphism classes of elliptic curves (\cite{McKean_Moll}). However, there are various reasons why we cannot quite content ourselves with the classical construction of this space. The main one, perhaps, being that the existence of a universal family of elliptic curves can only be obtained in the context of stacks, not manifolds. In this chapter, we will present a construction of the moduli stack of elliptic curves and show that it is in fact an analytic orbifold. This illustrates the need to go beyond the original definition of an orbifold because, in Thurston's definition of the quotient orbifold $[X/G]$, the group $G$ is supposed to act \textit{effectively} on the manifold $X$, but to prove that the moduli stack of elliptic curves is analytic, one ends up describing it as the quotient stack $[\fh/\SL(2;\Z)]$, where $\fh$ is the upper half-plane in $\C$ and $\SL(2;\Z)$ acts on it by homographic transformations. This action is not effective, and in fact its kernel is precisely $\{\pm I_2\}$, so $\PSL(2;\Z)$ acts effectively on $\fh$. While $\fh/\SL(2;\Z)$ and $\fh/\PSL(2;\Z)$ are homeomorphic topological spaces, the quotient stacks $[\fh/\SL(2;\Z)]$ and $[\fh/\PSL(2;\Z)]$ are not isomorphic, and only the first one is isomorphic to the moduli stack of elliptic curves, so the fact that $\fh/\PSL(2;\Z)$ is an orbifold in the sense of Thurston is not entirely satisfying here, even though it does get us the correct coarse moduli space. Resolving this issue by showing that $\fh/\SL(2;\Z)$ carries a canonical structure of non-effective orbifold is not quite good enough either, because of the difficulties that arise with the universal family. So the only way out of this seems to be considering the quotient stack $[\fh/\SL(2;\Z)]$ and showing that it is an orbifold in the stacky sense. Of course, there are deeper implications to this discussion, for instance the fact that the fundamental group of the moduli space of elliptic curves should be $\SL(2;\Z)$, not $\PSL(2;\Z)$, because elliptic curves always have an Abelian group structure, so in particular an inversion map $x\lmt -x$, which is the map coming from the element $-I_2\in \SL(2;\Z)$ acting trivially on $\fh$.

\smallskip

The example of the moduli space of elliptic curves illustrates how thinking about non-effective orbifolds as a particular class of stacks may be useful, even for classical problems. This suggests the need to go beyond Satake and Thurston's original definitions, using the input from algebraic geometry to enrich the differential-geometric perspective on the notion of orbifold. Note that this is not the only occurrence of such a phenomenon. While no evidence suggests that the theory of the étale fundamental group of schemes may have been of inspiration to Thurston (except perhaps his construction of the universal cover as an inverse limit in \cite[Proposition~13.2.4, pp.305--307]{Thurston_2002}), the analogy is clear, and in fact illustrates a common trend between the two fields, which is to include more maps in the definition of a covering space in order to have a richer notion of fundamental group. Indeed, consider for instance the orthogonal projection from the unit sphere in $\R^3$ to the closed unit disk in $\R^2\times\{0\}\subset\R^3$. In the usual topology of $\R^3$, this map is not a covering map (in fact, it is not even locally injective at points of the unit circle). But in algebraic geometry, it is the map induced on the set of closed points by the canonical projection $\P^1_\C\to\P^1_\R=\P^1_\C/\Gal(\C/\R)$, which is an étale morphism of schemes (in particular, the étale fundamental group of $\P^1_\R$ is not trivial, and since $\P^1_\C$ is simply connected, it is in fact exactly $\Gal(\C/\R)$). If we view this example through the orbifold prism, the orbifold fundamental group, in the sense of Thurston, of the quotient of $S^2\subset\R^3$ by the group generated by the reflection through the equator is indeed $\Z/2\Z$, just like in algebraic geometry. More generally, this leads to a richer notion of fundamental groups for all Klein surfaces (= quotients of Riemann surfaces by anti-holomorphic involutions), for which we refer for instance to \cite{Sch_CRM}. In analytic and algebraic geometry too, it is useful to allow orbifolds: the fundamental group of the quotient of $\C$ by the finite group $\mu_n(\C)\simeq\Z/n\Z$ acting by multiplication by a primitive $n$-th root of unity should be $\mu_n(\C)$, not the trivial group. As a matter of fact, the canonical projection $\C\lra \C/\mu_n(\C)$ is the typical case of a map which is ramified when seen as a morphism between smooth algebraic or analytic varieties (in which case it can be identified with the map $z\lmt z^n$), but which can be viewed as a covering map in the orbifold sense (see Example \ref{covers_of_quotient_stack}).

\smallskip

Recall that a complex-analytic elliptic curve is a pair $(\sC,e)$ where $\sC$ is a compact Riemann surface of genus $1$ and $e$ is a point in $\sC$. A morphism between two elliptic curves $(\sC_1,e_1)$ and $(\sC_2,e_2)$ is a holomorphic map $f:\sC_1\lra \sC_2$ such that $f(e_1)=e_2$. We shall denote by $f:(\sC_1,e_1)\lra (\sC_2,e_2)$ such a morphism. The main result that we wish to discuss in this chapter is stated as follows.

\begin{theorem_intro}[Theorem \ref{main_goal}]
The moduli stack of complex elliptic curves is a complex analytic stack, isomorphic to the orbifold $[\fh/\SL(2,\Z)]$, where $\fh:=\{z\in\C\ |\ \mathrm{Im}\,(z)>0\}$ is the hyperbolic plane and $$\SL(2,\Z):=\left\{\begin{pmatrix} a & c \\ b & d \end{pmatrix} : a,b,c,d \in \Z,\, ad-bc=1\right\}$$ is the so-called modular group, acting to the right on $\mathfrak{h}$ by homographic transformations 
$$ z \cdot \begin{pmatrix}a & c \\ b & d \end{pmatrix} = \frac{az+b}{cz+d}
\, .$$ 
This moduli stack admits the (non-compact) Riemann surface $\fh/\SL(2,\Z) \simeq \C$ as a coarse moduli space.
\end{theorem_intro}

In Section \ref{families_ell_curves}, we recall basic results on elliptic curves. Sections \ref{primer_stacks} and \ref{orbifolds_section} contain the formal definition of a stack and an explanation of what it means for a stack defined on the category of complex analytic manifolds to be analytic, and to be an orbifold. In Sections \ref{moduli_spaces_section} and \ref{stack_of_ell_curves}, we discuss the notion of moduli spaces for a  complex analytic stack and we prove that the stack of elliptic curves is an analytic orbifold. 

\medskip

This chapter is entirely expository and makes no claims to originality: its main goal is to be self-contained enough in order to be useful to young researchers who are entering the field and are interested in the interactions between differential and algebraic geometry, all the while celebrating the vitality and scope of Thurston's geometric vision.

\section{Complex elliptic curves}\label{families_ell_curves}

\subsection{Complex tori}

Let us start by recalling the definition of a complex elliptic curve (in the analytic setting). We recall that the (geometric) genus of a compact connected Riemann surface $\sC$ is the dimension of the complex vector space $H^{0}(\sC;\Omega^{1}_{\sC})$ which, as $\sC$ is $1$-dimensional and compact, is the space of holomorphic $1$-forms on $\sC$.

\begin{definition}\label{ell_curve_def}
An analytic \textit{complex elliptic curve}\index{elliptic curve} is a pair $(\sC,e)$ where $\sC$ is a compact Riemann surface of genus $1$ and $e$ is a point in $\sC$. A \textit{morphism of elliptic curves} between $(\sC,e)$ and $(\sC',e')$ is a holomorphic map $f: \sC \lra \sC'$ such that $f(e)=e'$.
\end{definition}

We now give the two sources of examples of complex elliptic curves that will be of interest to us: complex tori (Example \ref{cx_tori}) and smooth plane cubics (Example \ref{Legendre_family}).

\begin{example}\label{cx_tori}
Let us see $\C \simeq \R^2$ as a two-dimensional real vector space (equipped with its usual topology and additive group structure) and let $\Lambda$ be a discrete subgroup of $\C$ such that the quotient group $\C/\Lambda$ is compact (such a subgroup $\Lambda \subset \C$ will be called a (uniform) \textit{lattice}\index{lattice}). In particular, the group $\Lambda$ is a free $\Z$-module of rank $2$. Note that there is a canonical marked point in $\C/\Lambda$ in this case, namely $0\ \text{mod}\ \Lambda$, the neutral element for the group law. Moreover, there is a (unique) Riemann surface structure on $\C / \Lambda$ that makes the canonical projection $p: \C \lra \C/\Lambda$ holomorphic. A holomorphic $1$-form on $\C/\Lambda$ is necessarily closed, so it pulls back to an exact global $1$-form $f(z)\, dz$ on the simply connected Riemann surface $ \C $, with $f(z)$ holomorphic and $\Lambda$-periodic (in the sense that, if $ z \in \C $ and $\lambda \in \Lambda$, then $f(z+\lambda) = f(z)$; this holds because $dz$ is $\Lambda$-invariant for the action of $\Lambda$ on $\C$ by translation). Since $\C/\Lambda$ is compact, the holomorphic function $f$ has to be bounded, therefore constant by Liouville's theorem, and this proves that the space of holomorphic $1-$forms on $\mathbb{C}/\Lambda$ has complex dimension 1, hence that we have constructed a complex elliptic curve $(\C/\Lambda, 0\ \text{mod}\ \Lambda)$, which we will refer to as a ($1$-dimensional) \emph{complex torus}\index{complex torus}.
\end{example}

The complex tori of Example \ref{cx_tori} can be embedded onto non-singular complex projective sets using the Weierstrass $\wp$-function of the lattice $\Lambda$, namely 
$$
\wp_\Lambda(w) = \frac{1}{w^2} + \sum_{\lambda \in \Lambda\setminus\{0\}} \left( \frac{1}{(w-\lambda)^2} - \frac{1}{\lambda^2} \right)
$$
which, if we pick a basis $(\lambda,\mu)$ of $\Lambda$ as a $\Z$-module, satisfies the equation 
\begin{equation}\label{image_proj_embedding}
(\wp_\Lambda'(w))^2 = f_{(\lambda,\mu)}\big(\wp_\Lambda(w)\big),
\end{equation}
where 
$$
f_{(\lambda,\mu)}(x) = 4 (x - v_1) (x-v_2) (x-v_3)$$
with
$$
v_1 = \wp_\Lambda\left(\textstyle\frac{\lambda}{2}\right) \in \C,  v_2 = \wp_\Lambda\left(\textstyle\frac{\mu}{2}\right) \in \C \ \text{and} \ v_3 = \wp_\Lambda\left(\textstyle\frac{\lambda+\mu}{2}\right) \in \C\, ,
$$
thus inducing a holomorphic map 
\begin{equation}\label{embedding_cx_torus_in_proj_plane}
\begin{array}{rcl}
\C/\Lambda & \lra & \CP^2 \\
w \text{ mod } \Lambda & \lmt & \left[ \wp_\Lambda(w) : \wp_\Lambda'(w) : 1 \right]
\end{array}
\end{equation}
whose image, because of Equation \eqref{image_proj_embedding} and up to normalisation of the basis $(\lambda,\mu)$, is a smooth projective curve of the type described in Equation \eqref{eq123}.

\begin{example}\label{Legendre_family}
 Let $u\in \C\setminus\{0,1\}$ and consider the smooth projective curve
\begin{equation}\label{eq123}
\sC_{u} := \left\{ [x:y:z] \in\mathbb{CP}^{2} \ | \ y^{2}z=x(x-z)(x-u z)\right\}.
\end{equation}
By the genus-degree formula, the smooth cubic $\sC_{u}$ has genus $ g = \frac{ ( 3 - 1 ) ( 3 - 2 ) }{ 2 } = 1 $. And since the point $e:=[0:0:1]$  belongs to $\sC_{u}$ for all $u\in\mathbb{C}\setminus\{0,1\}$, we have defined a family of elliptic curves $(\sC_{u},e)_{u\in\C\setminus\{0,1\}}$, depending on the parameter $u$. This family is called the \textit{Legendre family}\index{Legendre family of elliptic curves} (see also Example \ref{Legendre_family_bis}).
\end{example}

Building on Example \ref{Legendre_family}, let us set $f_u(x) = x(x-1)(x-u)$. The classical theory of elliptic curves says that $\sC_u$ is a compactification of the plane cubic of equation $y^2 = f_u(x)$. Moreover, the map $(x,y) \lmt x$ induces a ramified double covering $p: \sC_u \lra \CP^1$, with four branch points, so an application of the Riemann-Hurwitz formula shows that $\sC_u$ is homeomorphic to $S^1 \times S^1$. If we pick a square root of $ f_u $, the differential form $\frac{dt}{\sqrt{f_u(t)}}$ on $\CP^1$ pulls back to a holomorphic $1$-form  $\omega_u := p^*\frac{dt}{\sqrt{f_u(t)}}$ that can be extended holomorphically to $\sC_u$. Then we can integrate this $1$-form along loops based at $e \in \sC_u$ and, since $\omega_u$ is closed, obtain a group morphism 
\begin{equation}\label{period_subgroup}
\Phi_{\omega_u}: 
\begin{array}{rcl}
\pi_1(\sC_u, e) & \lra & \C \\
\left[\gamma\right] & \lmt & \int_\gamma \omega_u
\end{array}
\end{equation}
whose image $\Lambda_{\omega_u} := \mathrm{Im}\,(\Phi_{\omega_u})$ is in fact a lattice in $\C$, called the \emph{period lattice}\index{lattice!period lattice}. This induces a holomorphic map 
\begin{equation}\label{isom_smooth_plane_cubic_cx_torus}
\begin{array}{rcl}
\sC_u & \lra & \C/\Lambda_{\omega_u} \\
x & \lmt & \int_{e}^{x} \omega_u \ \text{mod} \ \Lambda_{\omega_u}
\end{array}
\end{equation}
which is an isomorphism of elliptic curves $(\sC_u,e) \simeq (\C/\Lambda_{\omega_u}, 0\ \text{mod}\  \Lambda_{\omega_u})$. Note that if we choose a different non-zero holomorphic $1$-form $\omega$ on $\sC_u$, it will satisfy $\omega = \alpha \omega_u$ for some $\alpha\in\C^*$, because the space of such forms is $1$-dimensional. Consequently, the lattices $\Lambda_{\omega_u}$ and $\Lambda_{\omega}$ satisfy  $\Lambda_{\omega} = \alpha  \Lambda_{\omega_u}$, so multiplication by $\alpha$ in $\C$ induces an isomorphism of elliptic curves $(\C/\Lambda_{\omega_u}, 0\ \text{mod}\ \Lambda_{\omega_u}) \simeq (\C/\Lambda_\omega, 0\ \text{mod}\ \Lambda_\omega)$. 

\medskip

So, from the existence of the projective embedding \eqref{embedding_cx_torus_in_proj_plane} and the isomorphism \eqref{isom_smooth_plane_cubic_cx_torus}, we see that there is a bijection
$$
\big\{ 1\text{-dimensional complex tori} \big\} \, \big/ \, \text{isomorphism} \simeq \big\{ \text{smooth plane cubics} \big\} \, \big/ \, \text{isomorphism} \, .
$$

And as a matter of fact, every elliptic curve $(\sC,e)$ in the sense of Definition \ref{ell_curve_def} is of this form, meaning that it is isomorphic to a one-dimensional complex torus. One way to show this is to generalise the construction of the period subgroup \eqref{period_subgroup} to the abstract setting: choose a generator $\omega$ of the $1$-dimensional complex vector space $H^0(\sC;\Omega^{1}_{\sC})$ and show that the image of the group morphism $\Phi_\omega$ is a lattice $\Lambda_\omega \subset \C$ and that there is an isomorphism of elliptic curves $(\sC,e) \simeq (\C/\Lambda_\omega, 0 \ \text{mod}\ \Lambda_\omega)$. The issue here is that, with our definition of the genus, we do not yet know that $\sC$ is homeomorphic to $S^1\times S^1$, which makes it harder to prove that the period subgroup is a lattice in $\C$. The way to solve this difficulty is to use Abel's theorem, which gives a necessary and sufficient condition for a Weil divisor $D$ on a compact Riemann surface $\sC$ to be  a principal divisor, i.e.\ for $D$ to be the divisor of zeros and poles of a meromorphic function on $\sC$. More precisely, one can follow the approach in \cite[pp.163--173]{Forster} to show the following:
\begin{enumerate}
\item A Weil divisor $D\in \Div(\sC)$ is the divisor of zeros and poles of a meromorphic function $f$ on $\sC$ if and only if $D$ is of degree $0$ and there exists a chain $c \in C_1(\sC;\Z)$ such that $\partial c = D$ and $\int_c \omega = 0$ (Abel's theorem in the $g=1$ case, \cite[Theorem 20.7, p.163]{Forster}).
\item The period subgroup $\Lambda_\omega := \mathrm{Im}\,\Phi_\omega \subset \C$ is a lattice \cite[Theorem 21.4, p.168]{Forster}. This uses the previously stated version of Abel's theorem, to prove that $\Lambda_\omega$ is a discrete subgroup of $\C$.
\end{enumerate}
One gets in this way a $1$-dimensional complex torus $\mathrm{Jac}(\sC) := \C / \Lambda_\omega$, namely the Jacobian variety of the elliptic curve $(\sC,e)$. It depends on $\omega$, but if we choose a different basis $\alpha \omega$ of the complex vector space $H^0(\sC;\Omega^{1}_{\sC})$, we obtain an isomorphic complex torus, as we have already seen. Moreover, the Riemann-Roch theorem has the following consequence.

\begin{lemma}\label{gp_structure}
Let $(\sC,e)$ be an elliptic curve and let $\mathrm{Pic}_0(\sC) := \Div_0(\sC) / \Div_P(\sC)$ be the group of degree $0$ divisors modulo principal divisors. Then the map 
$$
\Psi:
\begin{array}{rcl}
\sC  & \lra & \mathrm{Pic}_0(\sC) \\
x & \lmt & [x] - [e]
\end{array}
$$
is bijective.
\end{lemma}

\begin{proof}
The map $\Psi$ is injective because, if $\Psi(x) = \Psi(x')$ with $x\neq x'$, then the Weil divisor $[x]-[x']$ would be principal, so there would exist a meromorphic function $f$ on $\sC$ with exactly one pole, of order $1$, and the corresponding holomorphic map $f:\sC \lra \CP^1$ would be an isomorphism, contradicting the fact that $\dim_\C H^0(\sC;\Omega^{1}_{\sC}) = 1$.
Moreover, the map $\Psi$ is surjective because every divisor $D$ of degree $0$ on $\sC$ is equivalent, modulo a principal divisor, to a divisor of the form $[x]-[e]$. Indeed, if we set $D' = D +[e]$, then $\deg(D') = \deg(D)+1= 1$ so, by the Riemann-Roch theorem, $\dim H^0(\sC;O_\sC(D')) \geq \deg(D') = 1$, meaning that there exists a non-zero meromorphic function $f$ on $\sC$ whose poles are bounded by $D'$, i.e.\ in terms of divisors, $(f) + D' \geq 0$. Since $(f)+D'$ is an effective Weil divisor of degree $1$, there exists a point $x$ in $\sC$ such that $(f)+D' = [x]$, hence $D = D' - [e] \sim_{(f)} [x] - [e]$.
\end{proof}

\begin{rem}
Lemma \ref{gp_structure} shows in particular that an elliptic curve $(\sC,e)$ has a natural structure of Abelian group, induced by the isomorphism $\Psi: (\sC,e) \overset{\simeq}{\lra} (\mathrm{Pic}_0(\sC),0)$ and the group structure on $\mathrm{Pic}_0(\sC)$. In particular, an elliptic curve always has non-trivial automorphisms (namely the inversion map $D \lmt -D$, seen as a morphism of marked Riemann surfaces).
\end{rem}

Finally, one can deduce from the above that a complex elliptic curve $(\sC,e)$ is isomorphic to a complex torus, namely the Jacobian variety $\mathrm{Jac}(\sC) := \C/\Lambda_\omega$, defined by the period lattice. This proves in particular that $\pi_1(\sC,e) \simeq \Lambda_\omega \simeq \Z^2$. So we see from the classification of compact connected orientable surfaces that $\sC$ is homeomorphic to $S^1\times S^1$.

\begin{thm}\label{isom_btw_ell_curve_and_its_Jacobian}
Let $(\sC,e)$ be a complex elliptic curve and let $\omega$ be a non-zero holomorphic $1$-form on $\sC$. Let $\Lambda_\omega \subset \C$ be the period lattice and let $\mathrm{Jac}(\sC) := \C/\Lambda_\omega$ be the associated complex torus, i.e.\ the Jacobian variety of $(\sC,e)$. Then the map 
$$
F:
\begin{array}{rcl}
\sC & \lra & \mathrm{Jac}(\sC) = \C/\Lambda_{\omega} \\
x & \lmt & \int_{e}^{x} \omega \ \mathrm{mod} \ \Lambda_{\omega}
\end{array}
$$ is an isomorphism of elliptic curves.
\end{thm}

We refer to \cite[Theorems 21.7 p.171 and 21.10 p.172]{Forster} for a complete proof of Theorem \ref{isom_btw_ell_curve_and_its_Jacobian}. Note that the map $F:\sC \lra \mathrm{Jac}(\sC)$ is the composition of the map $\Psi: \sC \lra  \mathrm{Pic}_0(\sC)$ of Lemma \ref{gp_structure} with the map 
$$
J:
\begin{array}{rcl}
\mathrm{Pic}_0(\sC) & \lra & \mathrm{Jac}(\sC) \\
D = \partial c & \lmt & \int_{c} \omega \ \text{mod} \ \Lambda_{\omega}
\end{array}
$$ which is well-defined by definition of the period lattice (i.e.\ independent of the choice of the chain $c\in C_1(\sC;\Z)$ such that $\partial c = D$), and injective in view of Abel's theorem. The surjectivity of $J$ is proven in \cite[Theorem 21.7 p.171]{Forster} and for the fact that the map $F:\sC \lra  \mathrm{Jac}(\sC)$ is holomorphic, we refer to \cite[Theorem 21.4 p.168]{Forster}.

\medskip

As an application, Theorem \ref{isom_btw_ell_curve_and_its_Jacobian} gives us the structure of the automorphism group of a complex elliptic curve and of a compact Riemann surface of genus $1$ without marked point. To obtain it, one just needs to observe that two complex tori $\C/\Lambda_1$ and $\C/\Lambda_2$ are isomorphic as complex elliptic curves if and only if there exists $\alpha\in\C^*$ such that $\alpha\Lambda_1 = \Lambda_2$, which can be deduced by lifting such an isomorphism to the universal cover $\C$ of these tori and using the classification of automorphisms of $\C$. In particular, the automorphism group of the elliptic curve $(\C/\Lambda,0)$ can be identified with the group of all $\alpha\in\C^*$ such that $\alpha\Lambda = \Lambda$. As $\Lambda$ is a discrete subgroup of $\C$, the set of such $\alpha$ is in fact a discrete subgroup of the circle group $S^1$, hence a finite rotation group, generated by a root of unity. If we do not ask to preserve the base point, then translations by an element $[z]\in\C/\Lambda$ also induce automorphisms of $\C/\Lambda$.

\begin{cor}\label{autom_of_elliptic_curves_and_genus_one_Riemann_surfaces}
Let $(\sC,e)$ be a complex elliptic curve. Then there exists a natural number $n\geq 1$ such that the automorphism group of the marked Riemann surface $(\sC,e)$ is
\begin{equation}\label{autom_of_elliptic_curves}
\Aut(\sC,e) \simeq \Z/n\Z.
\end{equation} Moreover, the automorphism group of the Riemann surface $\sC$ is the semi-direct product $$\Aut(\sC) \simeq \sC \rtimes \Aut(\sC,e)$$ where the group $\sC\simeq\mathrm{Pic}_0(\sC)$ acts on itself by translations and the structure of semi-direct product is given by the action of $\Aut(\sC,e)$ on $\sC$. In particular, the group $\Aut(\sC)$ is a $1$-dimensional compact analytic Lie group over $\C$, whose identity component is isomorphic to $\sC$.
\end{cor}

\begin{rem}
We will see in Corollary \ref{autom_of_elliptic_curves_and_genus_one_Riemann_surfaces_again} that, in fact, the natural number $n$ in Equation \eqref{autom_of_elliptic_curves} can only be $2$, $4$ or $6$ (never $1$!). In particular, we will give an alternate, more detailed proof of Corollary \ref{autom_of_elliptic_curves_and_genus_one_Riemann_surfaces}.
\end{rem}

So the conclusion we can draw from the present paragraph is two-folded: firstly, the classification of complex elliptic curves up to isomorphism can be reduced to the classification of complex tori, and secondly, such objects have non-trivial automorphisms as marked Riemann surfaces. As a consequence, even though we can always find, given a complex elliptic curve $(\sC,e)$, an isomorphism of elliptic curves with a complex torus $\C/\Lambda$, this isomorphism will never be unique, because we can compose it with an automorphism of $\C/\Lambda$. This is one of the reasons why constructing a moduli space of complex elliptic curves requires additional notions, namely framings, in order to get rid of non-trivial automorphisms.

\subsection{Framed elliptic curves}

When faced with a moduli problem for objects that have non-trivial automorphisms, an idea that has proven useful is to add an extra structure to these objects. This is because the corresponding notion of isomorphism will then be more restrictive (since the extra structure needs to be preserved, too), so in particular there will be fewer automorphisms. In the best case scenario, there is a group action on these objects with extra structure, and two such objects with extra structure lie in the same orbit if and only if they are isomorphic as objects without extra structure, so the moduli space of objects without extra structure can be constructed as a quotient of the moduli space of objects with extra structure (provided the latter moduli space actually exists, of course). We will now illustrate this more concretely in the case of elliptic curves. We begin with the definition of an extra structure.

\begin{definition}\label{framed_ell_curve}\index{elliptic curve!framed elliptic curve}\index{framing}
Let $(\sC,e)$ be a complex elliptic curve. A \textit{framing} of $(\sC,e)$ is a basis $([a],[b])$ of the free $\Z$-module $H_{1}(\sC,\mathbb{Z}) \simeq \pi_1(\sC,e) \simeq \Z^2$, with the additional property that 
\begin{equation}\label{positivity_cond}
\text{for all non-zero } \omega \in H^0\big(\sC;\Omega^1_{\sC}\big),\ \mathrm{Im}\, \textstyle \frac{\int_a \omega}{\int_b \omega} >0\,.
\end{equation} Since $H^0(\sC;\Omega^1_{\sC})$ is one-dimensional by assumption, it suffices to test Condition \eqref{positivity_cond} on a single non-zero holomorphic $1$-form $\omega$. A morphism of framed elliptic curves $$f: \big(\sC,e,([a],[b])\big) \lra \big(\sC',e',([a'],[b'])\big)$$ is a morphism of elliptic curves $f:(\sC,e) \lra (\sC',e')$ such that $f_*([a]) = [a']$ and $f_*[b] = [b']$, where $f_*: H_1(\sC;\Z) \lra H_1(\sC';\Z)$ is the group morphism induced by $f$ in homology.
\end{definition}

\begin{example}\label{framed_lattices}
If $(\sC,e) = (\C/\Lambda,0)$ for some lattice $\Lambda \subset \C$ and $(\lambda,\mu)$ is  a basis  of the $\Z$-module $\Lambda$ that satisfies the additional positivity condition $\mathrm{Im}\,\frac{\mu}{\lambda} >0$, then $(\lambda,\mu)$ defines a framing of $(\C/\Lambda,0)$ by letting $\hat{\lambda}$ and $\hat{\mu}$ be the $1$-cycles in $\C/\Lambda$ corresponding to the paths $t \lmt t \lambda$ and $t \lmt t \mu$ in $\C$. A lattice $\Lambda\subset \C$ equipped with such a \textit{direct basis} $(\lambda,\mu)$ is also called a \textit{framed lattice}\index{lattice!framed lattice}, and we shall sometimes identify a framing of $\Lambda$ in this sense with the induced framing of $\C/\Lambda$. So, when we write $(\C/\Lambda,0,(\lambda,\mu))$, we mean the framed elliptic curve $(\C/\Lambda,0,(\hat{\lambda},\hat{\mu}))$. Note that if a non-zero holomorphic $1$-form $\omega$ on $\sC$ is fixed, then we can define $(\lambda,\mu)$ from $([a],[b])$ by setting $\lambda:=\int_a \omega$ and $\mu:=\int_b \omega$, but without such an $\omega$, we cannot recover $(\lambda,\mu)$ from $(\hat{\lambda},\hat{\mu})$. Regardless, if $\Lambda' = \alpha\Lambda$ for some $\alpha\in\C^*$, then the map $z \lmt \alpha z$ induces an isomorphism of framed elliptic curves 
$$
\big(\C/\Lambda,0,(\lambda,\mu)\big) \simeq \big(\C/(\alpha\Lambda),0,(\alpha\lambda,\alpha\mu)\big).
$$ In particular, if we set $\tau := \frac{\mu}{\lambda}$, then $\tau$ is a complex number with positive imaginary part, i.e.\ $\mathrm{Im}\,(\tau) >0$, and there is an associated lattice $\Lambda(\tau):= \Z1 \oplus \Z \tau \subset \C$ with direct basis $(1,\tau)$ such that the framed elliptic curve $(\C/\Lambda, 0, (\lambda,\mu))$ is isomorphic to the framed elliptic curve $(\C/\Lambda(\tau), 0, (1,\tau))$.
\end{example}

In view of Theorem \ref{isom_btw_ell_curve_and_its_Jacobian} and Example \ref{framed_lattices}, a framed complex elliptic curve $(\sC,e,([a],[b]))$ is isomorphic to a framed complex torus of the form $(\C/\Lambda(\tau), 0, (1,\tau))$ where $$ \tau \in \fh := \{z\in\C \ |\ \mathrm{Im}\,(z) >0\}.$$ Indeed, it suffices to show that a complex elliptic curve admits a framing, which comes from the simple observation that if a basis $([a],[b])$ for $H_1(\sC;\Z)$ does not satisfy the positivity condition \eqref{positivity_cond}, then the basis $([b],[a])$ does. Note that, a priori, given a framed elliptic curve $(\sC,e,([a],[b]))$, there could be more than one $\tau\in \fh$ such that $(\sC,e,([a],[b])) \simeq (\C/\Lambda(\tau), 0, (1,\tau))$. But Lemma \ref{unique_tau} shows that this is in fact not the case. In particular, a framed elliptic curve of the form $(\C/\Lambda(\tau), 0, (1,\tau))$ has no non-trivial automorphisms. 

\begin{lemma}\label{unique_tau}
The framed complex tori $(\C/\Lambda(\tau), 0, (1,\tau))$ and $(\C/\Lambda(\tau'), 0, (1,\tau'))$ are isomorphic if and only if $\tau=\tau'$.
\end{lemma}

\begin{proof}
An isomorphism between such complex tori is induced by $z \lmt \alpha z$ for some $\alpha\in\C^*$, which respects the framing if and only if $\alpha=1$ and $\tau'=\tau$.
\end{proof}

\begin{rem}
It is not entirely apparent from the proof, but the positivity condition $\mathrm{Im}\,(\tau) >0$ is necessary if we want a uniqueness result such as the one in Lemma \ref{unique_tau} to hold. Indeed, if we had defined a framing of $\Lambda$ simply as a basis of the $\Z$-module $\Lambda$, then, even if we normalise the lattice so that the first vector in the basis be equal to $1$, two normalised framings of the complex torus $(\C/\Lambda(\tau), 0)$ do not have to be equal: of the two bases $(1,\tau)$ and $(\tau,1) \sim \big(1,\frac{1}{\tau}\big)$ of the lattice $\Lambda(\tau) = \Z1\oplus\Z\tau = \Z\tau\oplus\Z 1$, only the first one satisfies the required positivity condition.
\end{rem}

We then have the following classification result for framed elliptic curves, whose proof follows directly from the discussion above and Lemma \ref{unique_tau}.

\begin{thm}\label{Teich_space_of_ell_curves}
Let $(\sC,e,([a],[b]))$ be a framed elliptic curve. Then there exists a unique 
$$
\tau\in \fh = \big\{ z \in \C \ | \ \mathrm{Im}\,(\tau)>0 \big\}
$$ 
and a unique isomorphism of framed elliptic curves 
$$ \big(\sC,e,([a],[b])\big) \simeq \big(\C/\Lambda(\tau), 0, (1,\tau)\big). $$ In particular, framed elliptic curves have no non-trivial automorphisms: $$\Aut(\sC,e,([a],[b])) = \{\id_{\sC}\}.$$
\end{thm}

As a consequence, the points of the Riemann surface
$$ 
\fh := \{z\in\C \ |\ \mathrm{Im}\,(z) >0\} 
$$ represent isomorphism classes of framed elliptic curves. This parameter space of framed elliptic curves is called the \textit{Teichmüller space}\index{Teichmüller space of elliptic curves} of elliptic curves, for reasons that will be explained next. Note that Theorem \ref{Teich_space_of_ell_curves} constitutes an improvement on Theorem \ref{isom_btw_ell_curve_and_its_Jacobian}, since we can now compare two given elliptic curves in a more formal sense because (at least when they are framed) these elliptic curves are represented by points \textit{of a same space}, namely the upper half-plane $\fh$. The next step is to ask ourselves: when are two framed elliptic curves isomorphic as plain elliptic curves? This corresponds to forgetting the framing, so answering this question should also give us a sense of what the possible framings are on a given elliptic curve. More precisely, we will construct the set of isomorphism classes of elliptic curves as a quotient of the Teichmüller space $\fh$, so it will come equipped with a map from $\fh$ to it, and the fibres of that map will be in bijection with the possible framings (the map in question being the one that forgets the frame on a given framed elliptic curve).

\begin{thm}\label{modular_gp_action}
Let $(\sC,e,([a],[b]))$ and $(\sC',e',([a'],[b']))$ be framed elliptic curves and let $\tau, \tau' \in \fh$ be the corresponding points in the Teichmüller space of elliptic curves. Then the unframed elliptic curves $(\sC,e)$ and $(\sC',e')$ are isomorphic as elliptic curves if and only if there exists a matrix $$A=\begin{pmatrix} a & c \\ b & d \end{pmatrix} \in \SL(2;\Z) \text{ such that } \tau' = \frac{a\tau + b}{c\tau +d}\, ,$$ meaning that $\tau$ and $\tau'$ are related by a direct Möbius transformation with integral coefficients.
\end{thm}

\begin{proof}
In view of Theorem \ref{Teich_space_of_ell_curves}, the framed elliptic curves $(\sC,e,([a],[b]))$ and $(\sC',e',([a'],[b']))$ can be replaced by the framed elliptic curves $(\C/\Lambda(\tau),0,(1,\tau))$ and $(\C/\Lambda(\tau'),0,(1,\tau'))$ and the latter are isomorphic as unframed elliptic curves if and only if there exists $\alpha\in\C^*$ such that $\alpha \Lambda(\tau) = \Lambda(\tau')=\Z1\oplus\Z\tau'$. So there exist $a,b,c,d\in\Z$ such that $\tau' = \alpha(a\tau + b)$ and $1 = \alpha(c\tau +d)$. This implies that 
$$\alpha= \frac{1}{c\tau+d} \quad \text{and} \quad \tau' = \frac{a\tau +b}{c\tau+d}\,.$$ Applying the same argument to $ \Lambda(\tau)= \frac{1}{\alpha} \Lambda(\tau') $, we see that the (Möbius) transformation $z \lmt \frac{az + b}{cz+d}$ is bijective, so the matrix $\begin{pmatrix} a & c \\ b & d \end{pmatrix}$ is invertible, therefore has determinant $\pm1\in\Z$. And since it sends $\tau$  such that $\mathrm{Im}\,(\tau)>0$ to $\tau'$ such that $\mathrm{Im}\,(\tau')>0$, it must have determinant $+1$.
\end{proof}

So the framed elliptic curves $(\sC,e,([a],[b]))$ and $(\sC',e',([a'],[b']))$ become isomorphic when forgetting the frame if and only if the points $\tau, \tau' \in \fh$ lie in a same orbit of the $\SL(2;\Z)$-action on $\fh$ defined by 
\begin{equation}\label{action_mod_group}
z \cdot \begin{pmatrix}a & c \\ b & d \end{pmatrix} = \frac{az+b}{cz+d}\,.
\end{equation}
Note that this is a \textit{right action}. Thanks to this, we can construct a set whose points are in bijection with 
isomorphic classes of elliptic curves.

\begin{cor}\label{moduli_set}
The quotient set $\fh/\SL(2;\Z)$ is the set of isomorphism classes of elliptic curves and we have a forgetful map 
$$
F: \fh \lra \fh/\SL(2;\Z) 
$$ 
sending the isomorphism class of the framed elliptic curve $(\C/\Lambda(\tau),0,(1,\tau))$ to the isomorphism class of the elliptic curve $(\C/\Lambda(\tau),0)$. In particular, the orbit of a point $\tau\in\fh$ under the $\SL(2;\Z)$-action \eqref{action_mod_group} is in bijection with the set of all framings on the elliptic curve $(\C/\Lambda(\tau),0)$, and the stabiliser group of $\tau$ in $\SL(2;\Z)$ is isomorphic to the automorphism group of the elliptic curve $(\C/\Lambda(\tau),0)$.
\end{cor}

\subsection{The modular curve}\index{modular!modular curve}\label{modular_curve_section}

The \emph{modular curve} is the quotient $\fh/\SL(2;\Z)$, of the Teichmüller space $\fh$ by the action of the \textit{modular group}\index{modular!modular group} $\SL(2;\Z)$. Note that Corollary \ref{moduli_set} tells us that this quotient set should be the underlying set of the moduli space of elliptic curves. 

\medskip

In this section, we recall what the canonical structure of Riemann surface on $\fh/\SL(2;\Z)$ is. But first, the mere fact that it is a quotient set already gives us insight into the classification of complex elliptic curves: by studying the action of the modular group on the Teichmüller space (stabilisers and fundamental domain), we can determine what the possible automorphism groups of complex elliptic curves are, and what the quotient may be as a Riemann surface, which turns out to be intimately related to the structure of the modular group itself. Note for instance that the centre of the modular group is the subgroup $\{\pm I_2\} \simeq \Z/2\Z$ and that this subgroup acts trivially on the Teichmüller space. So every elliptic curve admits an order two automorphism. This is already apparent in Theorem \ref{isom_btw_ell_curve_and_its_Jacobian}, which endows every Riemann surface of genus $1$ with a marked point with a (non-canonical) structure of Abelian group: the associated order-two automorphism is just the inversion map $D\lmt (-D)$ in this Abelian group. And if the elliptic curve is given by a smooth cubic equation $y^2z = f(x,z)$ in the projective plane, then the involution is induced by the involution $[x:y:z] \lmt [x:(-y):z]$ of $\CP^2$. This also shows that the action of $\SL(2;\Z)$ on $\fh$ is not faithful, as the centre of the modular group acts trivially on $\fh$. Hence an induced action of $\PSL(2;\Z) := \SL(2;\Z) / \{\pm I_2\}$, which turns out to be faithful. The important observation here is that the orbits of the $\PSL(2;\Z)$-action are the same as those of the $\SL(2;\Z)$-action but the stabilisers are not the same. In particular, the $\PSL(2;\Z)$-stabiliser of a point $\tau\in\fh$ is \emph{not} isomorphic to the automorphism group of the elliptic curve $(\C/\Lambda(\tau),0)$. The following result is classical (see for instance \cite[Chapter VII.1 pp.77--79]{Serre73} or \cite{KConradSL2} for an exposition).

\begin{thm}\label{structure_of_SL_2_Z}
The modular group $\SL(2;\Z)$ is generated by the matrices 
$$
S = \left(\begin{array}{rr} 0 & -1 \\ 1 & 0 \end{array}\right)
\quad \text{and} \quad T 
=  \left(\begin{array}{rr} 1 & 1 \\ 0 & 1  \end{array}\right)
\, ,
$$ hence also by the matrices $S$ and $ R:=ST =  \left( \begin{array}{rr} 0 & -1 \\ 1 & 1 \end{array} \right) $. This induces the following presentations of $\SL(2;\Z)$ and $\PSL(2;\Z)$ by generators and relations:
$$ \SL(2;\Z) \ \simeq \ \big\langle \, S, R \ \big| \  S^2 = R^3,\ S^4 = R^6 = 1\, \big\rangle \  \simeq \ \Z/4\Z \ast_{\Z/2\Z} \Z/6\Z$$
and
$$ \PSL(2;\Z) \ \simeq \ \big\langle \, [S], [R] \ \big| \ [S]^2 = [R]^3 = 1 \big\rangle \ \simeq \ \Z/2\Z \ast \Z/3\Z.$$
\end{thm}

Theorem \ref{structure_of_SL_2_Z} is about the algebraic structure of $\SL(2;\Z)$ and $\PSL(2;\Z)$. The matrices $S$ and $R$ show that this group is generated by an element of order four and an element of order six, which themselves satisfy the relation $S^2=R^3$ (equal to $-I_2$). So $\SL(2;\Z)$ is the amalgamated coproduct of the two Abelian groups $\Z/4\Z$ and $\Z/6\Z$ over $\Z/2\Z$, with the non-trivial element of $\Z/2\Z$ being sent to $S^2$ in $\langle S \rangle \simeq \Z/4\Z$ and to $R^3$ in $ \langle R \rangle \simeq \Z/6\Z$. As a consequence, the group $\PSL(2;\Z) = \SL(2;\Z) / \{\pm I_2\}$ is generated by the classes $[S]$ and $[R]$, which are respectively of order $2$ and $3$ with no other relations between them, so $\PSL(2;\Z)$ is the coproduct of $\Z/2\Z$ and $\Z/3\Z$.  This implies for instance that a character $\chi:\SL(2;\Z)\lra \C^*$ has image contained in the subgroup generated by a fourth root of unity and a sixth root of unity, so $\mathrm{Im}\, \chi \subset \mu_{12}(\C)$. Geometrically, though, it is the action of the generators $S$ and $T$  that is the most convenient to visualise. Indeed, $S$ acts on $\fh$ via the Möbius transformation $[S]: z \lmt -\textstyle\frac{1}{z}$, which is a direct rotation of angle $\frac{\pi}{2}$ about $i$, and $T$ acts via $[T]: z \lmt z+1$, which is a translation. As a consequence, $R=ST$ acts via $[R]:z\lmt \frac{1}{1-z}$, which is a rotation of angle $\frac{\pi}{3}$ about $e^{i\pi/3}$. This leads to the usual construction of a fundamental domain for the action of $\SL(2;\Z)$ (or equivalently $\PSL(2;\Z)$, since the orbits are the same) on $\fh$.

\begin{thm}
Let 
$$
\cF := \big\{ \tau \in \fh \ | \ |\tau| \geq 1 \text{ and } \textstyle \frac{1}{2} \leq \mathrm{Re}\,(\tau) \leq \textstyle\frac{1}{2} \big\}\,.
$$ 
Then the following properties hold:
\begin{enumerate}
\item Every $\SL(2;\Z)$-orbit in $\fh$ intersects $\cF$ in at least one point and at most two points.
\item If an orbit intersects $\cF$ in two points, then these two points lie on the topological boundary $\partial \cF$ of $\cF$. 
\item The points of $\cF$ with non-trivial stabiliser in $\SL(2;\Z)$ are the points $i = e^{i\pi/2}$, whose stabiliser is the cyclic group of order $4$ generated by $S$, and the points $j = e^{i2\pi/3}$ and $-j^2 = -\frac{1}{j} = e^{i\pi/3} = j \cdot T$, whose stabilisers are cyclic groups of order $6$, generated respectively by $TRT^{-1} = TS$ and $R$.
\end{enumerate}
\end{thm}

Note that the point at infinity of $\CP^1=\C \cup \{\infty\}$ may be added to $\fh \subset \C$ and that the action of $\SL(2;\Z)$ on $\fh$ by biholomorphisms extends to a conformal action on $\hat{\fh}:=\fh\cup\{\infty\}$. The stabiliser of the point at infinity is then the infinite cyclic group generated by $T$. Algebraically, the point at infinity no longer corresponds to an elliptic curve (smooth plane cubic with a marked point) but to a nodal plane cubic with a marked point (given in an affine chart by an equation of the form $y^2 = x^2(x-1)$).

\begin{cor}\label{autom_of_elliptic_curves_and_genus_one_Riemann_surfaces_again}
Let $\tau\in\cF$ and let $\Lambda(\tau):=\Z1\oplus \Z\tau$ be the associated lattice in $\C$. Let $G := \Aut(\C/\Lambda(\tau),0)$ be the automorphism group of the elliptic curve $(\C/\Lambda(\tau),0)$. Then the following properties hold:
\begin{enumerate}
\item If $\tau\neq i, j, -j^2$, then $G\simeq\Z/2\Z$.
\item If $\tau=i$, then $G \simeq \Z/4\Z$.
\item If $\tau = j$ or $-j^2$, then $G \simeq \Z/6\Z$.
\end{enumerate}
\end{cor}

\noindent The shape of the fundamental domain $\cF= $ gives us an indication of what the quotient space $\fh/\PSL(2;\Z)$ might look like: a topological space homeomorphic to $\C$, and this identification can be used to put a Riemann surface structure on $\fh/\PSL(2;\Z)$. Thinking about $\fh/\PSL(2;\Z)$ in terms of $\cF$ also suggests that there are three special points, namely $i$, $j$ and $-j^2$, which are the points with non-trivial stabiliser in $\PSL(2;\Z)$, and this makes $\fh/\PSL(2;\Z)$ an orbifold in the sense of Thurston. We will now show that $\fh / \PSL(2;\Z)$ is also the quotient of a Riemann surface by the action of a \emph{finite} group (Theorem \ref{finite_pres}).

\begin{rem} We will recall in Corollary \ref{modular_orbi_curve} why $\SL(2;\Z)$ acts properly on $\fh$. By Theorem \ref{structure_of_SL_2_Z}, the group $\SL(2;\Z) $ is a finitely generated subgroup of $\GL(2;\C)$. So, by Selberg's Lemma, it contains a torsion-free normal subgroup of finite index, say $\pi \lhd \SL(2;\Z)$. The discrete group $\pi$, being a subgroup of $\SL(2;\Z)$, acts properly on $\fh$. Since $\pi$ is torsion-free, it must therefore act freely on $\fh$ (see Theorems \ref{orbifold_quotient_stacks} and  \ref{quotient_by_finite_and_discrete_group} for further details on that). So the canonical projection $\fh \lra \fh/\pi =: \tilde{M}$ is an analytic covering map between Riemann surfaces, and the topological space $M := \fh/\SL(2;\Z)$ is homeomorphic to $\tilde{M} / \Ga$, where $\Ga:= \SL(2;\Z) / \pi$ is a finite group. A priori, there could exist several pairs $(\tilde{M},\Ga)$ consisting of a Riemann surface $\tilde{M}$ and a finite group $\Ga$ acting on $M$ by automorphisms, such that $\tilde{M}/\Ga \simeq M$. An explicit such presentation is given in Theorem \ref{finite_pres}.
\end{rem}

\begin{thm}\label{finite_pres}
The set $M := \fh/\SL(2;\Z)$ of isomorphism classes of elliptic curves is in bijection with the set $\tilde{M}/\Ga$ where $\tilde{M} = \C\setminus \{0;1\}$ and $\Ga$ is the permutation group $\mathfrak{S}_3 = \left\langle \sigma, \tau\right\rangle $, generated by the permutations $\sigma = (1\ 2)$ and $\tau= (1\ 3)$ and acting to the right on $u \in \C\setminus\{0;1\}$ via $f_\sigma(u) := (1-u)$ and $f_\tau(u) := \textstyle \frac{1}{u}\,$.
\end{thm}

\begin{proof}
As seen after Example \ref{cx_tori}, for all $\tau\in\fh$, the framed elliptic curve $(\C/\Lambda(\tau),0,(1,\tau))$ can be embedded onto the smooth plane cubic of equation 
$$y^2 = 4 (x - v_1z) (x-v_2z) (x-v_3z)$$
in $\CP^2$, with pairwise distinct coefficients
$$
v_1 = \wp_{\Lambda(\tau)}\left(\textstyle\frac{1}{2}\right), \  v_2 = \wp_{\Lambda(\tau)}\left(\textstyle\frac{\tau}{2}\right) \ \text{and} \ v_3 = \wp_{\Lambda(\tau)}\left(\textstyle\frac{1+\tau}{2}\right)\in\C\,.
$$ Let us set $u := \textstyle \frac{v_3 - v_2}{v_1 -v_2} \in \C\setminus\{0;1\}$. This corresponds to sending $(v_1,v_2,v_3)$ to $(1,0,u)$ via the homographic transformation $z \lmt \textstyle \frac{z-v_2}{v_1-v_2}$, preserving $\infty\in\CP^1$. So we may also view $u$ as a point of $\CP^1\setminus\{0;1;\infty\}$, with limit cases $u=0,1,\infty$ corresponding respectively to $v_3 = v_2$, $v_3=v_1$ and $v_1=v_2$. The point is that, if $v_1$ and $v_2$ are tranposed by $\sigma = (1\ 2)\in\mathfrak{S}_3$, then $u$ is changed to $f_\sigma(u) = 1-u$, and if $v_1$ and $v_3$ are transposed by $\tau = (1\ 3)\in\mathfrak{S}_3$, then $u$ is changed to $f_\tau(u) = \textstyle \frac{1}{u}$. In fact, the group $\left\langle f_\sigma, f_\tau\right\rangle$ is the group of automorphisms of $\CP^1$ which preserve the subset $\{0;1;\infty\}$. Note that this defines a right action of $\mathfrak{S}_3$ on $\CP^1\setminus\{0;1;\infty\}$ because $f_{\sigma\tau} = f_\tau \circ f_\sigma$. Moreover, if $\tau' \in\fh$, then the corresponding $u'\in \CP^1\setminus\{0;1;\infty\}$ lies in the same $\mathfrak{S}_3$-orbit as $u$ if $\C/\Lambda(\tau) \simeq \C/\Lambda(\tau')$ as elliptic curves. In view of Theorem \ref{modular_gp_action}, this occurs if and only if $\tau' = \textstyle \frac{a\tau+b}{c\tau+d}$ for some $\textstyle\begin{pmatrix} a & c \\ b & d\end{pmatrix}\in \SL(2;\Z)$. Thus, we have indeed proved that 
\begin{equation}\label{quotient_RS}
\fh / \SL(2;\Z) \simeq \big( \C \setminus\{0;1\} \big) / \mathfrak{S}_3 \,.
\end{equation}
\end{proof}

\begin{rem}
The parameter $u\in\C\setminus\{0;1\}$ that appears in the proof of Theorem \ref{finite_pres} coincides with the parameter $u$ in the Legendre family of Example \ref{Legendre_family}. 
\end{rem}

Based on the bijection \eqref{quotient_RS}, we can identify this set with the Riemann surface $\C$ in two equivalent ways. Either by defining an $\SL(2;\Z)$-invariant surjection $\hat{j}: \fh\lra\C$ whose fibres are precisely the $\SL(2;\Z)$-orbits in $\fh$, or by defining a $\mathfrak{S}_3$-invariant surjection $\hat{\lambda}: \C\setminus\{0;1\} \lra \C$ whose fibres are precisely the $\mathfrak{S}_3$-orbits in $\C\setminus\{0;1\}$. The first method induces a bijection $j:\fh/\SL(2;\Z)\lra \C$ sending the $\SL(2;\Z)$-orbit of $\tau\in\fh$ to the $j$-invariant of the framed elliptic curve $(\C/\Lambda(\tau),0,(1,\tau))$. The latter is constructed via the theory of modular forms and defined by converting Equation \eqref{image_proj_embedding} into the equation 
$$
\wp^{\prime 2}(z)=4 \wp^3(z)-g_2(\tau) \wp(z)-g_3(\tau)
$$
where $g_2(\tau) = 60 G_4(\tau)$, $g_3(\tau) = 140 G_6(\tau)$ and $G_k(\tau)$ is the Eisenstein series $$ \sum_{(m,n)\in\Z^2, (m,n)\neq (0,0)} \frac{1}{(m\tau+n)^k}\,,$$ then setting $j(\tau) = \textstyle\frac{1728\, g_2(\tau)^3}{g_2(\tau)^3 - 27 g_3(\tau)^2}$. For the former approach and the definition of a bijection $\lambda : \C\setminus\{0;1\} / \mathfrak{S}_3 \lra \C$, recall that we want a surjective map $\hat{\lambda}: \C\setminus\{0;1\} \lra \C$ with generic fibres having six elements, and that is invariant under the transformations $f_\sigma: u\lmt (1-u)$ and $f_\tau: u \lmt \textstyle \frac{1}{u}$. This leads to $\hat{\lambda}(u) := \textstyle \frac{(1 + u(u-1))^3}{(u(u-1))^2}$, whose numerator is a polynomial of degree $6$ which is non-zero when $u=0$ or $u=1$. It can be checked that, indeed, $\hat{\lambda} \circ f_\sigma = \hat{\lambda}$ and $\hat{\lambda} \circ f_\tau = \hat{\lambda}$. Finally, note that if we reduce the cubic equation $y^2 = x(x-1)(x-u)$ to an equation of the form $y^2 = 4x^3 -g_2x - g_3$, we find that 
$g_2 = \textstyle \frac{\sqrt[3]{4}}{3}(1 + u (u-1))$ and $g_3 = \textstyle \frac{1}{27}(u+1)(2 u^2-5 u+2)$, so the $j$-invariant of the elliptic curve defined by the Legendre equation $y^2 = x(x-1)(x-u)$ is equal to $256 \times \hat{\lambda}(u)$.

\section{Stacks}\label{primer_stacks}

\subsection{Prestacks}

Prestacks formalise the notion of \emph{families} of geometric objects by building upon the intuition that a family is a special kind of map, one that you should be able to pull back. For instance, if a group $G$ acts on a space $X$, there should be:
\begin{enumerate}
\item For all space $S$, a notion of family of $G$-orbits in $X$, parameterised by $S$, and a notion of isomorphism of such families.
\item For all morphism $f:T\lra S$ and all family $P$ parameterised by $S$, a family $f^*P$ parameterised by $T$, such that there are natural isomorphisms $\id_S^*P\simeq P$ and $(f\circ g)^*P \simeq g^*(f^*P)$ for all morphism $g:U\lra T$.
\end{enumerate}

\begin{example}\label{quotient_prestack}
If $G$ is a Lie group acting by automorphisms on a manifold $X$, a family parameterised by $S$ of $G$-orbits in $X$ is defined to be a pair $(P,u)$ where $P$ is a principal $G$-bundle on $S$, and $u:P\lra X$ is a $G$-equivariant map. In particular, the image of $u$ is a union of $G$-orbits in $X$, and the intersection of the fibre of $u$ above a point $x\in X$ with a $G$-orbit in $P$, if non-empty, is in bijection with the stabiliser of $x$ in $G$. Moreover, if $f:T\lra S$ is a morphism, then $f^*P$ is indeed a principal $G$-bundle on $T$, equipped with a $G$-equivariant map to $X$:

\begin{equation*}
\begin{tikzcd}
f^*P \ar[r] \ar[d] & P \ar[r, "u"] \ar[d] & X\\
T \ar[r, "f"] & S & 
\end{tikzcd}
\end{equation*}

\medskip

An isomorphism between two families $(P,u)$ and $(Q,v)$ \emph{parameterised by the same space $S$}, is an isomorphism $\lambda:P\overset{\simeq}{\lra} Q$ of principal $G$-bundles over $S$, such that $v\circ\lambda = u$. Note that the definitions above make sense when $X=\bullet$ (a single point), with $G$ acting trivially on it.
\end{example}

We are now ready to give the formal definition of a prestack. Since we are mainly interested in the moduli stack of elliptic curves, our spaces $S$ will be \emph{complex analytic manifolds}. We shall denote by $\An$ the category of such spaces.

\begin{definition}\label{def_prestack}\index{prestack}
A \emph{prestack on $\An$} is a $2$-functor $\M: \An^{\mathrm{op}} \lra \text{Groupoids}$, meaning an assignment,
\begin{itemize}
\item for all complex analytic manifold $S\in\An$, of a groupoid $\M(S)$;
\item for all morphism $f:T\lra S$ in $\An$, of a pullback functor $f^{\,*}_\M:\M(S)\lra \M(T)$, often denoted simply by $f^*$;
\item for all pair of composable morphisms $g:U\lra T$ and $f:T\lra S$ in $\An$, of a commutative diagram 
\begin{equation}\label{identification_of_composed_pullback}
\begin{tikzcd}
\M(S) \ar[r, "f^*"] \ar[d, "(f\circ g)^*"] & \M(T) \ar[d, "g^*"] \\
\M(U) \ar[r, "\Phi^{\,\M}_{f,g}"] & \M(U) 
\end{tikzcd}
\end{equation}
\end{itemize}
such that:
\begin{enumerate}
\item $\id_S^* = \id_{\M(S)}$, and
\item for all morphism $h:V\lra U$ in $\An$, the following diagram commutes:
\begin{equation}\label{compatibility_of_pullbacks}
\begin{tikzcd}
& \M(S) \ar[rd, "f^*"] \ar[ld, "(f\circ g)^*"] \ar[ddd, dashed, "(f\circ g\circ h)^*"]&\\ 
\M(U) \ar[d, "h^*"] & & \M(T) \ar[d, "(g\circ h)^*"] \\
\M(V) \ar[rd, "\Phi^{\,\M}_{f\circ g,h}"] & & \M(V) \ar[ld, "\Phi^{\,\M}_{f,g\circ h}"] \\
& \M(V) &
\end{tikzcd}
\end{equation}
In other words, the dashed arrow $(f\circ g\circ h)^*:\M(S) \lra \M(V)$ is well-defined.
\end{enumerate}
\end{definition}
In practice, the identification $(f\circ g)^*\simeq_{\Phi^{\,\M}_{f,g}} g^*\circ f^*$ in Diagram \eqref{identification_of_composed_pullback} is often canonical: this was for instance the case in Example \ref{quotient_prestack} where, for all principal bundle $P\lra S$, it follows from the construction of the pullback bundle that there is indeed a canonical isomorphism $(f\circ g)^*P = g^*(f^*P)$. Likewise $h^*((f\circ g)^*P) = (g\circ h)^* ( f^* P)$ , which means that Condition \eqref{compatibility_of_pullbacks} is automatically satisfied in this case. So Example \ref{quotient_prestack} indeed defines a prestack, usually denoted by $[X/G]$ and called the \emph{quotient stack} of $X$ by $G$. If $\M$ is a prestack and $S$ is a manifold, the groupoid $\M(S)$ is called the groupoid of $S$-points of $\M$. For instance, the groupoid of $S$-points of $[X/G]$ is $$[X/G](S) := \left\{
\begin{tikzcd}
 P \ar[r, "u"'] \ar[d] & X\\
S & 
\end{tikzcd} 
\begin{array}{l}
\text{with}\ u:P\to X\ \text{a}\ G\text{-equivariant map}\\
\text{and}\ P\to S\ \text{a\ principal}\ G\text{-bundle}
\end{array}
\right\}.$$

When $X=\bullet$ is a point, the prestack $[\bullet/G]$ is usually denoted by $BG$, and called the \emph{classifying stack\index{stack!classifying stack} of $G$}. The case when $G=\{1\}$ is also interesting, for it enables us to see manifolds as prestacks (Example \ref{manifolds_as_prestacks}). Finally, as we shall see in Section \ref{def_of_stack_of_ell_curves}, families of elliptic curves parameterised by a complex analytic manifold $S\in\An$ indeed define a prestack.

\begin{example}\label{manifolds_as_prestacks}
A manifold $X$ defines a prestack, denoted by $\underline{X}$, by setting $\underline{X}(S):=\mathrm{Mor}(S,X)=\{u:S\lra X\}$, with morphisms between $u:S\lra X$ and $u':S \lra X$ given by 
$$ \Mor_{ \underline{ X }( S ) }( u, u' ) := 
\left\{ 
\begin{array}{cl}
\{\id_S\} & \text{if}\ u=u' \, , \\
\emptyset & \text{if}\ u \neq u' \, .
\end{array}
\right.
$$ Note that this is the standard way in which groupoids generalise sets: the groupoid associated to a set is the category whose objects are the elements of the set and the only morphisms are the identity morphisms. The pullback functor corresponding to $f:T\lra S$ is given by $f^*u=u\circ f$, which is indeed a functor given the way morphisms are defined in $ \underline{X}( S ) $ and $ \underline{X}( T ) $. As we shall see in Lemma \ref{Yoneda}, this defines a fully faithful functor $X\lmt \underline{X}$ from manifolds to stacks. A prestack $\X$ for which there exists a manifold $X$ such that $ \X \simeq \underline{X} $ is called a \emph{representable stack}\index{stack!representable stack}\index{representable!stack}. By Lemma \ref{Yoneda}, a manifold $X$ satisfying $\underline{X}\simeq \X$ is defined up to canonical isomorphism.
\end{example}

To conclude this subsection, we give the definition of a morphism of prestacks.

\begin{definition}
Let $\M,\N:\An^{\mathrm{op}}\lra\text{Groupoids}$ be prestacks on $\An$. A \emph{morphism of prestacks} $F:\M\lra \N$ is a a natural transformation between the $2$-functors $\M$ and $\N$, meaning an assignment for all space $S\in\An$, of a functor $F_S:\M(S)\lra \N(S)$ such that, for all morphism $f:T\lra S$ in $\An$, the following diagram is $2$-commutative (i.e.\ commutative up to a natural transformation in the category $ \N( T ) $)
\begin{equation}\label{2_comm_diag}
\begin{tikzcd}
\M(S) \ar[r, "F_S"] \ar[d, "f^{\,*}_\M"] & \N(S) \ar[d, "f^{\,*}_\N"]\\
\M(T) \ar[r, "F_T"] & \N(T)
\end{tikzcd}
\end{equation}
and moreover, for all $g:U\lra T$ in $\An$, we have $F_U(\Phi^{\,\M}_{f,g}) = \Phi^{\,\N}_{F_U(f),F_U(g)}$ as functors from $\N(U)$ to $\N(U)$. We shall denote by $\mathcal{PS}t$ the $2$-category of prestacks.
\end{definition}

\begin{example}\label{towards_Yoneda}
If $X$ is a manifold, morphisms of prestacks $F\in\Mor_{\mathcal{PS}t}(\underline{X},\M)$ correspond to objects $P_X\in\M(X)$. Indeed, the morphism $F$ is entirely determined by the object $P_X:=F_X(\id_X)\in \M(X)$ and the fact that, as $\id_X \circ u = u$, the morphism $F$ must satisfy, for all $S\in\An$, 
$$F_S:\begin{array}{rcl}
\underline{X}(S)=\Mor_{\An}(S,X) & \lra & \M(S)\\
u & \lmt & u^*_\M P_X
\end{array}\, .
$$ This sets up an equivalence of categories $\Mor_{\mathcal{PS}t}(\underline{X},\M) \simeq \M(X)$, obtained by sending $F$ to $P_X:=F_X(\id_X)$. For instance, if $\M=BG:=[\bullet/G]$, the groupoid $BG(S)$ is the category of principal 
$G$-bundles on $S$ and, given a principal bundle $P\lra X$, there is a natural functor
$$F_S: \begin{array}{rcl}
\Mor_{\An}(S,X) & \lra & BG(S)\\
u & \lmt & u^* P
\end{array}
$$ and the $2$-commutativity of Diagram \eqref{2_comm_diag} holds because of the canonical identification $(u\circ f)^*P \simeq f^*(u^*P)$ for all $f:T\lra S$.
\end{example}

Unwrapping the details of Example \ref{towards_Yoneda}, we get the following stacky version of the Yoneda lemma.

\begin{lemma}\label{Yoneda}
Let $X,Y$ be objects in $\An$ and let $\underline{X}, \underline{Y}$ be the corresponding prestacks. Then $$\Mor_{\mathcal{PS}t}(\underline{X},\underline{Y}) \simeq \Mor_{\An}(X,Y).$$ In particular, $\underline{X}\simeq \underline{Y}$ in $\mathcal{PS}t$ if and only if $X\simeq Y$ in $\An$.
\end{lemma}

\subsection{Stacks}

There is an obvious analogy between Definition \ref{def_prestack} and that of a presheaf on $\An$. Based on that analogy, we can define stacks as prestacks that satisfy a descent condition with respect to open coverings of objects $S\in\An$. The only subtlety is that we have to glue not only morphisms but also objects. In what follows, given a prestack $\M$ on $\An$, a morphism $f:T\lra S$ in $\An$, and a morphism $\phi:A\lra B$ in $\M(S)$, we denote by $A|_T$ and $B|_T$ the objects $f^*A$ and $f^*B$ of $\M(T)$, and by $\phi|_T:A|_T\lra B|_T$ the morphism $f_\M^*\phi$ in $\M(T)$.

\begin{definition}\label{def_stack}\index{stack}
A \emph{stack} on $\An$ is a prestack $\M:\An^{\mathrm{op}}\lra\text{Groupoids}$ such that, for all $S\in\An$ and all open covering $(S_i)_{i\in I}$ of $S$, the following two conditions are satisfied:
\begin{enumerate}
\item (Gluing of morphisms). For all pair of objects $A,B\in\M(S)$ and all family of morphisms $$\phi_i:A|_{S_i}\lra B|_{S_i}\ \text{in}\ \M(S_i)$$ such that, for all $ i, j$ in $I$, $ \phi|_{ S_i \cap S_j } = \phi_j|_{ S_i \cap S_j } $, there exists a unique morphism $\phi:A\lra B$ in $\M(S)$ such that, for all $i\in I$, $\phi|_{S_i}=\phi_i$.
\item (Gluing of objects). For all family of objects $(A_i)_{i\in I} \in \prod_{i\in I}\M(S_i)$ equipped with isomorphisms $$\phi_{ij}: A_j|_{S_i\cap S_j} \overset{\simeq}{\lra} A_i|_{S_i\cap S_j}\ \text{in}\ \M(S_i\cap S_j)$$ satisfying the cocycle conditions $\phi_{ii} = \id_{A_i}$ and $\phi_{ij}\circ\phi_{jk} = \phi_{ik}$ in $\M(S_i\cap S_j\cap S_k)$ for all triple $(i,j,k)$, there exists an object $A\in\M(S)$ together with isomorphisms $\phi_i:A|_{S_i}\lra A_i$ in $\M(S_i)$, such that, for all $(i,j)$, one has $ \phi_i \circ\phi_j^{-1} = \phi_{ij}$ in $\M(S_i\cap S_j)$.
\end{enumerate}
A morphism of stacks $F:\M\lra\N$ is a morphism of prestacks from $\M$ to $\N$. We shall denote by $\mathcal{S}t$ the $2$-category of stacks.
\end{definition}

\begin{rem}
It is perhaps not easy to note at first, but an important feature of the definition of stacks is that objects are required to glue but without any requirement of uniqueness: in Point (2) of Definition \ref{def_stack}, the object $ A $ constructed from the $ A_i $ is ``unique up to isomorphism", by a direct application of Condition (1), but the isomorphisms $ \phi_i: A|_{ S_i } \lra A_i $ are not unique, which is the familiar condition from the theory of fibre bundles and is in fact what allows objects with non-trivial automorphisms to define stacks. We will for instance study the stack of elliptic curves from that point of view in Section \ref{stack_of_ell_curves}.
\end{rem}

\begin{prop}\label{prestack_of_ell_curves_is_a_stack}
The prestacks $[X/G]$, $BG$ and $\underline{X}$ introduced in Examples \ref{quotient_prestack} and \ref{manifolds_as_prestacks} are stacks on $\An$.
\end{prop}

\begin{proof}
The prestacks $BG$ and $\underline{X}$ are special cases of the prestack $[X/G]$. We leave the proof that the prestack $[X/G]$ is a stack as an exercise (it follows from the fact that principal $G$-bundles on $S\in\An$, and morphisms between them, can be constructed by gluing).
\end{proof}

In line with the intuition that stacks are designed to formalise the idea of families as mathematical objects that can be \emph{pulled back} (prestack) and \emph{glued} (stacks), the $2$-category of stacks admits fibre products. 

\begin{definition}\label{fibre_product}\index{fibre product of stacks}
Let $F:\M\lra \X$ and $G:\N\lra \X$ be morphisms of stacks on $\An$. Let us denote by $\M \times_\X \N$ the $2$-functor sending a space $S\in\An$ to the groupoid  whose objects are
\begin{equation}\label{fibre_product_def}
\Mor_{\mathcal{S}t}(\underline{S},\M\times_\X\N) := 
\left\{
(\alpha,\beta,\phi)\quad \Bigg|\quad \substack{\displaystyle{\alpha:\underline{S}\lra \M,}\\ \displaystyle{\beta: \underline{S}\lra \N,}\\ \displaystyle{\phi:\X\overset{\simeq}{\lra} \X}}\ \text{such that}
\begin{tikzcd}
& \X \ar[dd, "\phi", "\simeq"'] \\
\underline{S} \ar[ur, "F\circ \alpha"] \ar[dr, "G\circ \beta"'] & \\
& \X 
\end{tikzcd}\ 
\text{commutes.}
\right\}
\end{equation} and whose morphisms between the objects $(\alpha,\beta,\phi)$ and $(\alpha',\beta',\phi')$ are given by triples $(\lambda,\mu,\psi)$ of morphisms $\lambda:\M\lra\M$, $\mu:\N\lra\N$ and $\psi:\X\overset{\simeq}{\lra}\X$ such that $F\circ(\lambda\circ\alpha) = \psi \circ (F\circ\alpha)$ and $G\circ(\mu\circ\beta) = \psi \circ (G\circ\beta)$ as morphisms from $\underline{S}$ to $\X$, and $\phi'\circ\psi=\psi\circ\phi$ as morphisms from $\X$ to $\X$ (i.e.\ $(\lambda,\mu,\psi)$ induces an isomorphism between diagrams of the form appearing in \eqref{fibre_product_def}).
\end{definition}

Definition \ref{fibre_product} is not easy to digest, because it involves a lot of hidden compatibility conditions. It is instructive to compare it to Thurston's definition of the fibre product of two orbifolds, which appears in the course of the proof of Proposition~13.2.4 in \cite[pp.305--307]{Thurston_2002}. If we think of the groupoid $\Mor_{\mathcal{PS}t}(\underline{S},\M\times_\X\N)$ as families parameterised by $S$, we get the following more concrete description:
\begin{equation}\label{fibre_product_alternate_def}
(\M \times_\X \N)(S) := \big\{(A,B,\phi_S)\ |\ A\in\M(S), g\in\N(S)\ \text{and}\ \phi_S:F_S(A)\overset{\simeq}{\lra} G_S(B)\ \text{in}\ \X(S)\big\}
\end{equation} with morphisms $(A,B,\phi_S)\lra (A',B',\phi'_S)$ being pairs $(u_S,v_S)$ of morphisms $u_S:A\lra A'$ in $\M(S)$ and $v_S:B\lra B'$ in $\N(S)$ making the following diagram commute in $\X(S)$:
\begin{equation*}
\begin{tikzcd}
F_S(A) \ar[r, "\phi_S", "\simeq"'] \ar[d, "F_S(u_s)"] & G_S(B) \ar[d, "G _S(v_S)"] \\
F_S(A') \ar[r, "\phi'_S", "\simeq"'] & G_S(B')
\end{tikzcd}
\end{equation*} Let us show how we can use \eqref{fibre_product_alternate_def} as a definition of the fibre product $\M\times_\chi\N$. We simply observe that, because of the fact that, for all morphism $f:T\lra S$ in $\An$, we have $F_T(A|_T) = F_S(A)|_T$ and $G_T(B_T)=G_S(B)|_T$ in $\X(T)$, we must add the compatibility condition that $\phi_T=\phi_S|_T$ as functors from $\X(T)$ to $\X(T)$, to ensure that we have a well-defined pullback morphism 
$$
f^{\ *}_{\M\times_\X\N} : 
\begin{array}{rcl}
(\M\times_\X\N)(S) & \lra & (\M\times_\X\N)(T)\\
(A,B,\phi_S) & \lmt & (A|_T, B|_T,\phi_T)
\end{array}\,.
$$ This compatibility condition, which needs to be added here, is encapsulated in Definition \ref{fibre_product} as the fact that the arrow $\phi:\X\lra\X$ appearing in Diagram \eqref{fibre_product_def} is a morphism of prestacks.

\begin{example}\label{towards_atlas_for_BG}
Let us compute the fibre product of a morphism $\underline{M}\lra BG$ with the morphism $\underline{\bullet}\lra BG$, where $M\in\An$ and $\bullet$ is a single point. By Example \ref{towards_Yoneda}, the first morphism corresponds to a principal $G$-bundle $P\lra M$, while the second morphism corresponds to the $G$-bundle $G\lra\bullet$. Then, by definition of the stacks $\underline{M}$, $\underline{\bullet}$ and $\underline{M}\times_{BG}\bullet$, and using the fact that a principal bundle is trivial if and only if it admits a global section, we have:
\begin{eqnarray*}
(\underline{M}\times_{BG}\underline{\bullet})(S) & \simeq & \left\{ (u,\phi)\ |\ u:S\lra M\ \text{and}\ \phi\ \text{is\ an\ isomorphism}\ u^*P\lra S\times G\right\}\\ 
& \simeq & \{ (u,\sigma)\ |\ u:S\lra M\ \text{and}\ \sigma\ \text{is\ a\ section\ of}\ u^*P\}.
\end{eqnarray*} But, by the universal property of $u^*P$, giving a section $\sigma:S\lra u^*P$ is equivalent to giving a morphism $f:S\lra P$.
$$
\begin{tikzcd}
S \ar[ddr, bend right=20, "\id_S"'] \ar[rrd, bend left=20, "f"] \ar[rd, "\exists!\,\sigma", dashed]& & \\
& u^*P \ar[r] \ar[d] & P \ar[d] \\
& S \ar[r] & M
\end{tikzcd}
$$ So $(\underline{M}\times_{BG}\underline{\bullet})(S) \simeq \Mor(S,P)$, which proves that the fibre product $\underline{M}\times_{BG}\bullet$ is isomorphic to the stack $\underline{P}$, where $P\lra M$ is the principal $G$-bundle defining the given morphism $\underline{M}\lra BG$. In particular, the stack $\underline{M}\times_{BG}\bullet$ is representable by the manifold $P$.
\end{example}

\begin{prop}
The $2$-functor $\M\times_\X\N$ defined in \eqref{fibre_product_def} is a prestack and it comes equipped with two morphisms $p_\M:\M\times_\X\N\lra \M$ and $p_N:\M\times_\X\N\lra \N$ making the following diagram a $2$-cartesian square:
\begin{equation*}
\begin{tikzcd}
\M\times_\X\N \ar[r, "p_\N"] \ar[d, "p_\M"] & \N \ar[d, "G"]\\
\M \ar[r, "F"] & \X
\end{tikzcd}
\end{equation*}
\end{prop}

\begin{proof}
We deduce from the discussion following Definition \ref{fibre_product} that $\M\times_\chi\N$ is a prestack. By \eqref{fibre_product_def}, the prestack $\M\times_\chi\N$ satisfies the universal property of a $2$-fibre product for all prestacks of the form $\underline{S}$ where $S\in\An$. Therefore, it satisfies it for all prestacks on $\An$.
\end{proof}

We omit the proof that the prestack $\M \times_\X \N$ is a stack. Note that if $\underline{M}, \underline{N}$, and $\underline{X}$ are representable stacks, the stack $\underline{M}\times_{\underline{X}} \underline{N}$ is not representable in general, because $M\lra X$ and $N\lra X$ do not admit a fibre product in $\An$ in general.

\subsection{Analytic stacks} An analytic stack is a stack $\X$ on $\An$ that admits a \emph{presentation} (or \emph{atlas}) $p:\underline{X}\lra \X$, where $\underline{X}$ is the stack associated to a space $X\in\An$ (henceforth, we shall often write simply $X$ for the stack $\underline{X}$ associated to the space $X\in\An$). Before we can define more precisely what a presentation is, we need to introduce the following notion.

\begin{definition}\label{representable_morphism}\index{representable!morphism}
A morphism of stacks $F:\M\lra\X$ is called \emph{representable} if, for all morphism $S\lra \M$ with $S\in\An$, there exists a complex analytic manifold $M_S\in\An$ such that $S\times_\X \M \simeq M_S$.
\end{definition}

In other words, for all $S\lra \X$, the morphism of stacks $F:\M\lra\X$ base-changes to a morphism of complex analytic manifolds $f_S:M_S\lra S$, as per the following $2$-Cartesian square:
\begin{equation}\label{base_change_rep_morphism}
\begin{tikzcd}
S\times_\X\M \simeq M_S\ar[d, "f_S"] \ar[r] & \M \ar[d, "F"] \\
S \ar[r] & 	\X
\end{tikzcd}
\end{equation}

\begin{example}\label{representability_of_the_can_morphism_from_X_to_X_mod_G}
The computation of $M\times_{BG}\bullet$ in Example \ref{towards_atlas_for_BG} shows that the canonical morphism $\bullet\lra BG$ is representable. Since the pullback of the trivial bundle $(X\times G)\lra X$ to an arbitrary manifold $S$ mapping to $X$ is trivial, the same proof works for the canonical morphism $X\lra [X/G]$, corresponding to the trivial bundle $X\times G$ (with $G$-equivariant map from $X\times G$ to $X$ given by the action map $(X\times G)\lra X$, where $G$ acts on $X$ to the right). Note that the forgetful morphism $BG\lra\bullet$ (sending a principal $G$-bundle $P\lra S$ to the constant map $S\lra\bullet$) is not a representable morphism as soon as $\bullet\times_\bullet BG \simeq BG$ is not a representable stack.
\end{example}

Working with representable morphisms enables one to carry over certain local properties of morphisms from $\An$ to $\mathcal{S}t$. Most prominently for us, the following one.

\begin{definition}\label{submersion_between_stacks}
Let $\M$ and $\X$ be stacks over $\An$. A morphism of stacks $F:\M\lra\X$ is called a \emph{submersion} if it is representable and, for all morphism of stacks $S\lra \X$ with $S\in\An$, the induced morphism $f_S:M_S\lra S$ appearing in Diagram \eqref{base_change_rep_morphism} is a submersion in $\An$.
\end{definition}

Intuitively, this is a meaningful definition because, for manifolds, a morphism $f:Y\lra X$ is a submersion in $\An$ if and only if for all open coverings $X = \cup_{i\in I} X_i$ and $Y=\cup_{j\in J} Y_j$, the induced morphisms $f_{ij}: Y_j\cap f^{-1}(X_i) \lra X_i$ are submersions. Next, we define surjective morphisms of stacks (which bears an analogy to the notion of surjective morphism of sheaves).

\begin{definition}\label{surjective_morphism_of_stacks}
Let $\M$ and $\X$ be stacks over $\An$. A morphism $F:\M\lra\X$ is said to be \emph{surjective} if, for all $S\in\An$ and all $P\in \X(S)$, there is an open covering $(S_i)_{i\in I}$ of $S$ such that $P|_{S_i}$ is in the essential image of the functor $F_{S_i}:\M(S_i)\lra \X(S_i)$, i.e.\ there exists a family $(P_i)_{i\in I}$ of objects $P_i\in \M(S_i)$ and, for all $i\in I$, an isomorphism $F_{S_i}(P_i) \simeq P|_{S_i}$ in $\X(S_i)$.
\end{definition}

We can now define a \emph{presentation} (or \emph{atlas of an analytic stack})\index{atlas} of a stack $\X$ on $\An$ as a surjective representable morphism $p:X\lra \X$ which is a submersion and whose domain is a manifold $X\in\An$. This leads to the following definition of an analytic stack.

\begin{definition}\label{analytic_stack_def}\index{stack!analytic stack}
An \emph{analytic stack} is a stack $\X$ on $\An$ equipped with an analytic manifold $X\in\An$ and a morphism $p:X\lra\X$ such that: 
\begin{enumerate}
\item $p$ is representable (Definition \ref{representable_morphism}).
\item $p$ is surjective (Definition \ref{surjective_morphism_of_stacks}).
\item $p$ is a submersion (Definition \ref{submersion_between_stacks}).
\end{enumerate}
A morphism between analytic stacks $\M\lra \X$ is a morphism of stacks from $\M$ to $\X$.
\end{definition}

The same definition works for a differentiable stack $\X$; it is a stack on the category of differentiable manifolds which is equipped with a representable surjective morphism $p:X\lra \X$ such that $X$ is a differentiable manifold and $p$ is a submersion.

\begin{example}\label{quot_stacks_are_analytic}
It is straightforward to check that, if $X\in\An$, the identity morphism $\id_X$ gives a presentation $X\lra \underline{X}$. So $\underline{X}$ is an analytic stack in a canonical way. In fact, we could also take, as a presentation for $\underline{X}$, the analytic manifold $\sqcup_{i\in I}U_i$, where $(U_i)_{i\in I}$ is an analytic atlas of $X$ in the traditional sense. Let us now prove that if $G$ is a complex Lie group acting on an analytic manifold $X$, then the quotient stack $[X/G]$ admits a canonical structure of analytic stack, given by the morphism $X\lra[X/G]$ defined by the trivial principal $G$-bundle on $X$. We saw in Example \ref{representability_of_the_can_morphism_from_X_to_X_mod_G} that this morphism is representable. It is surjective because principal $G$-bundles on a manifold $S$ are locally trivial. And it is a submersion because the projection map of a principal bundle $P\lra S$ is a submersion. In particular, the classifying stack $BG$ of a complex Lie group $G$ admits a canonical structure of analytic stack.
\end{example}

Note that a fibre product of analytic stacks is a stack over $\An$ but not necessarily an analytic stack.

\section{Orbifolds}\label{orbifolds_section}

We will eventually define an analytic orbifold as an analytic stack admitting an open covering by quotient subststacks of the form $[U/\Gamma]$, where $U$ is an analytic manifold and $\Gamma$ is a finite group acting on $U$ by automorphisms. Before we can do that, we need to define open substacks.

\begin{definition}\index{substack}
Let $\X$ be a stack on $\An$. A \emph{substack} of $\X$ is a pair $(\Y,F)$ where $\Y$ is a stack on $\An$ and $F:\Y\lra \X$ is a morphism of stacks. If $\X$ is an analytic stack and the substack $\Y$ is also analytic, we will call $\Y$ an \emph{analytic substack} of $\X$.
\end{definition}

\begin{example}\label{quotient_substack}
Let $G$ be a complex Lie group acting to the right on an analytic manifold $X$ and let $Y\subset X$ be an \emph{open} submanifold of $X$ that is invariant under the action of a subgroup $H\subset G$. The $H$-equivariant map $Y\hookrightarrow X$ induces a canonical morphism of stacks $F:[Y/H]\lra [X/G]$. Explicitly, the morphism $F$ is defined, for all $S\in\An$, by the functor 
\begin{equation}\label{ext_struct_gp}
F_S: 
\left(
\begin{tikzcd} P \ar[r] \ar[d] & Y\\ S & \end{tikzcd}
\right) 
\lmt 
\left(
\begin{tikzcd} (P\times_H G) \ar[r] \ar[d] & Y\cdot G \underset{\text{\tiny{open}}}{\subset} X\\ S & \end{tikzcd}
\right)
\end{equation} sending a principal $H$-bundle $P$ on $S$ equipped with an $H$-equivariant map to $Y$ to the extension of structure group $P\times_H G$, equipped with the induced $G$-equivariant map to $Y\cdot G\subset X$ . This morphism $F$ makes $[Y/H]$ a substack of $[X/G]$. Note that $[X/G](S)$ is in general strictly larger than $[Y/H](S)$, since it is possible that not all principal $G$-bundles on $S$ admit a reduction of structure group on $H$. Also, the functor $F_S$ is in general neither full nor faithful: given two principal $G$-bundles on $S$ admitting a reduction of structure group to $H$, a morphism between them does not necessarily come from a morphism between the two reductions and, even if it does, the morphism it comes from is not necessarily unique (for instance, a given principal $G$-bundle may admit non-isomorphic reductions of structure group to $H$). Note that if we study the analogous situation with $Y$ \emph{closed} in $X$, we need to be careful with the manifold structure on $Y\cdot G$ in order to get a morphism $Y\cdot G\lra X$ in $\An$ (which may or may not be a \emph{closed} immersion).
\end{example}

It is easy to make sense of the notion of open substack for representable morphisms (Definition \ref{open_substack_def}). And for analytic stacks, it suffices to check the condition with respect to an atlas (Proposition \ref{check_opennness_on_atlas}).

\begin{definition}\label{open_substack_def}
Let $\X$ be a stack on $\An$. A substack $F:\Y\lra\X$ is called an \emph{open substack}\index{substack!open substack} of $\X$ if:
\begin{enumerate}
\item $F$ is representable (Definition \ref{representable_morphism}).
\item For all morphism of stacks $M\lra \X$ with $M\in\An$, the morphism $f_M: \Y\times_\X M \lra M$ is an open morphism in $\An$.
\end{enumerate}
\end{definition}

Note that Condition 2, which says that $f_M: \Y\times_\X M \lra M$ is open in $\An$, makes sense because the stack $\Y\times_\X M$ is representable by a manifold (in this case because $F$ is assumed to be a representable morphism; see Remark \ref{rep_morph_for_analytic_stack} for further comments).

\begin{example}\label{open_substacks_in_quotient_stack}
Let us retake Example \ref{quotient_substack}. We will show that, if $Y\subset X$ is open and $G/H$ is a \emph{discrete} topological space, then $[Y/H]$ is an open substack of $[X/G]$. This will apply in particular when $G$ is a discrete group (see also Example \ref{covers_of_quotient_stack}). Firstly, we have to show that the morphism \eqref{ext_struct_gp} is representable and secondly, that it is open (this is where we shall use that $G/H$ is discrete). For the first part, we compute the fibre product $[Y/H]\times_{[X/G]} M$ for all morphism of stacks $M\lra [X/G]$ with $M\in\An$. Such a morphism is given by a principal $G$-bundle $P$ over $M$, equipped with a $G$-equivariant morphism $P\lra X$ in $\An$. Likewise, an arbitrary morphism $S\lra[Y/H]$ with $S\in\An$ is given by a principal $H$-bundle $Q$ over $S$, equipped with an $H$-equivariant morphism $Q\lra Y$ in $\An$. So, given a morphism $u:S\lra M$ in $\An$, a morphism of stacks $S\lra [Y/H]\times_{[X/G]} M$ is given by an isomorphism of principal $G$-bundles between $u^*P$ and $Q\times_H G$ over $S$, i.e.\ a reduction of structure group from $G$ to $H$ for the principal $G$-bundle $u^*P\lra S$. It is well-known that reductions of structure group are in bijection with sections of the associated fibre bundle $u^*P\times_G (G/H) \simeq u^*(P\times_G (G/H))$ over $S$. Then, as in Example \ref{towards_atlas_for_BG}, such sections are in bijection with morphisms of analytic manifolds from $S$ to $P \times_G (G/H)$, which proves that the canonical morphism $F:[Y/H]\lra[X/G]$, defined in \eqref{ext_struct_gp} by extension of structure group and inclusion of $Y\cdot G$ in $X$, is representable (in particular, when $M=X$ and $P=X\times G$, then $P\times_G(G/H) = X\times (G/H)$). Since we now have a $2$-Cartesian diagram 
$$
\begin{tikzcd}
 P \times_G (G/H)\simeq \left[Y/H\right]_M \ar[r] \ar[d] & \left[Y/H\right] \ar[d] \\
 M \ar[r, "P"'] & \left[X/G\right]
\end{tikzcd}
$$ we see that the morphism of stacks $[Y/H]\lra[X/G]$ pulls back to the fibre bundle $P\times_G (G/H) \lra M$ in $\An$, which is indeed an open map \textit{since the fibre, which is $G/H$, is discrete}. 
\end{example}

\begin{rem}
Note that, if we take $H=G$ in Example \ref{open_substacks_in_quotient_stack}, then the canonical morphism $[Y/G]\lra [X/G]$ is an open \emph{embedding} in the sense that, for all $M\in\An$, the morphism $[Y/G]_M \lra M$ is an open embedding in $\An$ (this is obvious because $[Y/G]_M \simeq M$ when $H=G$). 
\end{rem}

\begin{prop}\label{check_opennness_on_atlas}
Let $\X$ be a stack on $\An$ and assume that $\X$ is an \emph{analytic} stack with presentation $p:X\lra\X$. A substack $F:\Y\lra\X$ is an \emph{open substack} of $\X$ if:
\begin{enumerate}
\item $F$ is representable.
\item The morphism $f_X: \Y\times_\X X \lra X$ is an open morphism in $\An$.
\end{enumerate}
\end{prop}

\begin{proof}
If $(\Y,F)$ is an open substack of the analytic stack $\X$, Conditions (1) and (2) of Proposition \ref{check_opennness_on_atlas} are indeed met. To prove the converse, it suffices to show that, if $f_X:Y_X\lra X$ is open, then $f_M:Y_M\lra M$ is open for all $M\in\An$, where, for all $S\in\An$, we use the notation $Y_S\in\An$ for the representable stack $\Y\times_\X S$. To prove what we need to, note that, given a morphism of stacks $M\lra \X$, we have a commutative diagram
$$
\begin{tikzcd}
Y_X\times_X X_M \ar[d, "\text{open}"] \ar[r,"\simeq"] & \Y\times_\X X_M \ar[r] \ar[d, "\text{open}"] & \Y\times_\X M \ar[d] \ar[r] & \Y \ar[d]\\
X_M \ar[r, equal] & X_M \ar[r, "\overset{\text{\tiny{surjective}}}{\text{\tiny{submersion}}}"] & M \ar[r] & \X
\end{tikzcd}
$$ where the first vertical arrow is open because $f_X:Y_X\lra X$ is open by the assumption made on $F:\Y\lra \X$ and since the property of being open is invariant by base change. The commutativity of the left square then implies that the second vertical arrow is open. Since $F:\Y\lra \X$ is a representable morphism of stacks, the stack $\Y\times_\X M$ is representable by a manifold, and since the middle square is a pullback diagram, the morphism $Y\times_\X X_M \lra Y\times_\X M$ is also a surjective submersion in $\An$. In particular, it is open and surjective. So the fact that $Y\times_\X X_M\lra X_M$ is open implies that $f:\Y\times_\X M \lra M$ is indeed open.
\end{proof}

We leave it as an exercise to show that, if $F:\Y\lra\X$ is an open substack of an analytic stack $\X$, then $\Y$ is an analytic stack.

\begin{rem}\label{rep_morph_for_analytic_stack}
Some comments are in order. For the sake of simplicity, we have omitted certain aspects of the definition of an analytic stack or of an open substack. For instance, in Definition \ref{analytic_stack_def}, it is often required that, for a stack $\X$ over $\An$ to be an analytic stack, \emph{every} morphism $M\lra\X$ with $M$ in $\An$ should be representable. So Condition (1) of Definition \ref{analytic_stack_def} becomes unnecessary. But of course it then becomes more difficult to show that a given stack over $\An$ is an analytic stack. In practice, one uses the fact that the following two conditions are equivalent (we refer to \cite[Remark 8.1.6 p.170]{Olsson2016} for a proof of this):
\begin{enumerate}
\item Every morphism $M\lra\X$ with $M$ in $\An$ is representable (meaning that, for all $N\lra\X$ with $N\in\An$, the stack $M\times_\X N$ is representable by a manifold).
\item The diagonal morphism $\Delta:\X\lra\X\times\X$ is representable.
\end{enumerate} 
So the more official version of Definition \ref{analytic_stack_def} is that a stack $\X$ on $\An$ is called an analytic stack\index{stack!analytic stack} if the diagonal morphism $\Delta:\X\lra\X\times\X$ is representable and there exists a manifold $X\in\An$ and a (necessarily representable) surjective submersion $p: X \lra \X$. The stack $\X$ is called \emph{separated}\index{stack!separated stack} if the representable morphism $\Delta$ is closed.
\end{rem}

In the same way that we defined what it means to be an open substack of a stack $\X$, we can define what it means to be a \emph{cover} of such a stack. Comments similar to Proposition \ref{check_opennness_on_atlas} and Remark \ref{rep_morph_for_analytic_stack} apply when $\X$ is an \emph{analytic} stack, but we omit them here and only give the elementary definition.

\begin{definition}\label{stacky_cover_def}
Let $\X$ be a stack on $\An$. A morphism of stacks $F:\Y\lra\X$ is called a \emph{cover} or \emph{covering stack}\index{stack!covering stack} of $\X$ if:
\begin{enumerate}
\item $F$ is representable.
\item For all morphism of stacks $M\lra \X$ with $M\in\An$, the morphism $f_M: \Y\times_\X M \lra M$ is a cover in $\An$.
\end{enumerate} Such a cover is called a \emph{finite cover} if, for all $M$, the covering map $f_M$ is a finite cover in $\An$. 
\end{definition}

\begin{example}\label{covers_of_quotient_stack}
If $G$ is a \emph{discrete group} acting by automorphisms on the space $X\in\An$, then the canonical morphism $X\lra [X/G]$ (defined by the trivial cover $(X\times G)\lra X$) is a covering stack of $[X/G]$. Indeed, if $M\in\An$ and $P\lra M$ is the principal $G$-bundle associated to a morphism $M\lra [X/G]$, then $X\times_{[X/G]} M$ is isomorphic to $P$ (see Examples \ref{towards_atlas_for_BG} and \ref{representability_of_the_can_morphism_from_X_to_X_mod_G}), and this is indeed a cover of $M$ because $G$ is discrete. A subgroup $H< G$ then defines an intermediate cover $X\lra [X/H] \lra [X/G]$. And if $Y$ is a $G$-equivariant cover of $X$, then $[Y/G]\lra [X/G]$ is a cover in the stacky sense. Mixing these two examples together, we get the following commutative diagrams, all of whose arrows are covers (the first vertical arrow being a $G$-equivariant cover in $\An$).
$$
\begin{tikzcd}
Y \ar[r] \ar[d, "H< G\ \Circlearrowright"'] & \left[Y/H\right] \ar[r] \ar[d] & \left[Y/G\right] \ar[d] \\
X \ar[r] & \left[X/H\right] \ar[r] & \left[X/G\right] \\
\end{tikzcd}$$ The proof of these claims follows from the same techniques as the ones used in Example \ref{open_substacks_in_quotient_stack}. And for a concrete example, the group $\mu_n(\C)$ of $n$-th roots of unity acts (non-freely) on $\C$ by multiplication by $e^{i\frac{2\pi}{n}}$, and we have a covering stack (or \emph{stacky cover}) $\C\lra[\C/\mu_n(\C)]$. The same construction works if we replace $\C$ by an open disk $D(0,\epsilon)\subset \C$.
\end{example}

Note that a covering stack is in particular an open morphism of stacks because the covering maps $f_M: \Y\times_\X M \lra M$ are all open morphisms in $\An$. We are now ready to give the definition of an orbifold. The intuition if that an analytic stack is an orbifold if it admits an \emph{open} covering of the form $[U/\Gamma]$, where $U\in\An$ and $\Gamma$ is a \emph{finite} group acting on $U$ by automorphisms.

\begin{definition}\label{orbifold_def}\index{orbifold}
An \emph{analytic orbifold} is an analytic stack $\X$ (Definition \ref{analytic_stack_def}) equipped with a family of substacks of the form $\big(F_i:[U_i/\Gamma_i]\lra \X\big)_{i\in I}$ such that:
\begin{enumerate}
\item For all $i\in I$, the pair $(U_i,\Gamma_i)$ consists of a manifold $U_i\in\An$ and a \emph{finite} group $\Gamma_i$ acting on $U_i$ by automorphisms.
\item For all $i\in I$, the substack $F_i:[U_i/\Gamma_i]\lra \X$ is an open substack (Definition \ref{open_substack_def}).
\item The morphism of stacks $F:\sqcup_{i\in I} [U_i/\Gamma_i] \lra \X$ is surjective (Definition \ref{surjective_morphism_of_stacks}). 
\end{enumerate}
\end{definition}

\noindent The stack $\sqcup_{i\in I} [U_i/\Gamma_i]$ appearing in Definition \ref{orbifold_def} is defined, for all $S\in\An$, by $\big(\sqcup_{i\in I} [U_i/\Gamma_i]\big)(S) = \sqcup_{i\in I} \big([U_i/\Gamma_i](S)\big)$. An ultimately equivalent but slightly more abstract definition of an analytic orbifold would be that of an analytic stack equipped with a (surjective) finite cover $F:M\lra \X$ with $M\in\An$. For instance, all manifolds have canonical orbifold structures. Either way, the main source of examples is obtained as an application of the following result, which is due to Thurston in the case when the group $G$ acts effectively on the manifold $X$ \cite[Proposition~13.2.1, p.302]{Thurston_2002}.

\begin{thm}\label{orbifold_quotient_stacks}
Let $X$ be an analytic manifold and let $G$ be a discrete group acting to the right on $X$ by automorphisms. For all $x\in X$, we denote by $$G_x:=\{g\in G\ |\ x\cdot g = x\}$$ the stabiliser of $x$ in $G$. Assume that the following properties are satisfied:
\begin{enumerate}
\item For all $x\in X$, there exists an open neighbourhood $U_x$ of $x$ such that 
\begin{equation*}
g\in G_x \Rightarrow (U_x\cdot g) \subset U_x\quad \text{and}\quad g\notin G_x \Rightarrow (U_x \cdot g) \cap U_x = \emptyset.
\end{equation*}
\item For all $x\in X$, the group $G_x$ is finite.
\end{enumerate}
Then the quotient stack $[X/G]$ is an orbifold.
\end{thm}

\begin{proof}
Since, for all $x\in X$, the subset $U_x\subset X$ is $G_x$-invariant and open in $X$, we know from Example \ref{open_substacks_in_quotient_stack} that the quotient stack $[U_x/G_x]$ defines an open substack of $[X/G]$. And since the $(U_x)_{x\in X}$ form an open covering of $X$, the canonical morphism of stacks $F: \sqcup_{x\in X} [U_x/G_x] \lra [X/G]$ is surjective. Indeed, if $S\in\An$ and $(P,u)$ is a pair consisting of a principal $G$-bundle $P$ on $S$ and a $G$-equivariant map $u:P\lra X$, we can take an open covering $(S_i)_{i\in I}$ of $S$ such that, for all $i$, the principal $G$-bundle $P_i:=P|_{S_i}$ is trivial. Let us fix a trivialisation $P_i\simeq (U_i\times G)$ in which $u|_{P_i}:P_i\lra X$ is the action map, and choose $x\in X$ such that $x\in u(P_i)$. Set then $Q_{i,x} := (S_i \times G_x)$. By Condition (1), an element $\xi\in u^{-1}(U_x)\cap P_i$ satisfies $g\cdot\xi\in u^{-1}(U_x)\cap P_i$ if and only if $g\in G_x$, so $u^{-1}(U_x)\cap P_i \simeq Q_{i,x}$. So $P_i$ lies in the essential image of the canonical functor $$F_{x,S_i}: [U_x/G_x](S_i) \lra [X/G](S_i),$$ hence also in the image of $F_S:\sqcup_{x\in X} [U_x/G_x](S_i)\lra [X/G](S_i)$, which proves that the morphism $F: \sqcup_{x\in X} [U_x/G_x]\lra [X/G]$ is surjective in the sense of Definition \ref{surjective_morphism_of_stacks}. Finally, since each $G_x$ is finite by Condition (2), we have indeed shown that $[X/G]$ is an orbifold.
\end{proof}

\begin{example}\label{quotient_by_finite_and_discrete_group}
If $G$ is finite to begin with, we can simply use the finite cover $X\lra[X/G]$ to prove that $[X/G]$ has a canonical orbifold structure. For instance, the quotient stack $[\C/\mu_n(\C)]$ of Example \ref{covers_of_quotient_stack} is canonically an orbifold. We will see in Corollary \ref{modular_orbi_curve} that the action of the discrete group $\SL(2;\Z)$ on the upper half-plane $\fh:=\{z\in\C\ |\ \text{Im}\,z>0\}$ satisfies the conditions of Theorem \ref{orbifold_quotient_stacks}, so the quotient stack $[\fh/\SL(2;\Z)]$ is also an orbifold, in a canonical way.
\end{example}

\begin{rem}
Note that we are not assuming that the action of $G$ on $X$ is effective, only that it has finite stabilisers. In particular, if $G$ is a \emph{finite} group, then the classifying stack $BG=[\bullet/G]$ is an orbifold. When the action of the group $G$ on the space $X$ is effective, the group $G$ can be identified with a subgroup of $\Aut(X)$ and the quotient stack $[X/G]$ is called an \emph{effective orbifold}.\index{orbifold!effective orbifold}
\end{rem}

To construct orbifolds, it suffices to find examples of actions of discrete groups on manifolds for which Conditions (1) and (2) of Theorem \ref{orbifold_quotient_stacks} are satisfied. An important tool to do that is the following result. We state it for a real Lie group $H$, but we are mostly interested here in the case where $H$ is a \emph{Hermitian group}, meaning that the symmetric space $X:=K\backslash H$ has a canonical structure of complex analytic (in fact, Kähler) manifold, in which case the quotient stack $[X/G]$ is indeed an analytic orbifold.

\begin{thm}\cite[Corollary~3.5.11]{Thurston_1997} \label{proper_actions_of_discrete_groups_by_isometries}
Let $H$ be a real Lie group whose Lie algebra admits an $H$-invariant positive definite scalar product. Fix a maximal compact subgroup $K\subset H$ and let us endow the homogeneous space $X:=K\backslash H$ with its canonical manifold structure. Let $G\subset H$ be a discrete subgroup of $H$. Then the canonical right action of $G$ on $X$ satisfies Conditions (1) and (2) of Theorem \ref{orbifold_quotient_stacks}. In particular, the quotient stack $[X/G]$ has a canonical orbifold structure.
\end{thm}

\begin{cor}\label{modular_orbi_curve}
Since the upper half-plane $\fh=\{z\in\C\ |\ \mathrm{Im}\,z>0\}$ can be described as the homogeneous space $\mathbf{SO}(2)\backslash\SL(2;\R)$, the quotient stack $[\fh/\SL(2;\Z)]$ has a canonical orbifold structure.
\end{cor}

Let us make a few remarks on the proof of Theorem \ref{proper_actions_of_discrete_groups_by_isometries}. Remarks \ref{properly_discont_actions} and \ref{Dirichlet_domain} show that \emph{proper, isometric} actions of \emph{discrete subgroups of Lie groups} on Riemannian manifolds give rise to orbifolds. Together with Remark \ref{isometric_action}, they can be formalised into a proof of Theorem \ref{proper_actions_of_discrete_groups_by_isometries} and Corollary \ref{modular_orbi_curve}.

\begin{rem}\label{properly_discont_actions}
What is actually proven in \cite[Corollary~3.5.11]{Thurston_1997} is that, if $G\subset H$ is a discrete subgroup of the Lie group $H$, the action of $G$ on $K\backslash H$ (or, as a matter of fact, on all manifold $X$ on which $H$ acts transitively with \emph{compact} stabilisers $H_x$) satisfies the following condition:
\begin{itemize}
\item[(P)] For all compact subset $C\subset X$, the set $\{g\in G\ |\ C\cdot g \cap C \neq \emptyset\}$ is finite.
\end{itemize} An (effective) action of a group $G$ by homeomorphisms on a locally compact space $X$ is called \emph{properly discontinuous}\index{properly discontinuous action} if it satisfies Property (P); see \cite[Definition~3.5.1.(v), p.153]{Thurston_1997}. If $G$ is endowed with the \emph{discrete} topology, this is equivalent to asking that the continuous map $X\times G\lra X\times X$ given by $(x,g)\lmt (x,x\cdot g)$ be \emph{proper}. So giving a properly discontinuous action on a locally compact space is equivalent to giving a proper, continuous action of a discrete group on that space \cite[Exercise~3.5.2, p.154]{Thurston_1997}. Note that, for such an action, the stabilisers $G_x$ are finite, so Condition (2) of Theorem \ref{orbifold_quotient_stacks} is indeed satisfied. 
\end{rem}

\begin{rem}\label{Dirichlet_domain} To prove Theorem \ref{proper_actions_of_discrete_groups_by_isometries}, we still need to show that if a discrete group $G$ acts continuously and properly on a manifold $X$, then, as stated in the proof of \cite[Proposition~13.2.1, p.302]{Thurston_2002}, for all $x\in X$, there exists an open neighbourhood $U_x$ of $x$ in $X$ which is $G_x$-invariant and disjoint from its translates by elements of $G$ that are not in $G_x$ (i.e.\ Condition (1) of Theorem \ref{orbifold_quotient_stacks} is also satisfied). In fact, to prove Theorem \ref{proper_actions_of_discrete_groups_by_isometries}, it suffices to do it in the case when there is a $G$-invariant distance $d$ on $X$ (i.e.\ the group $G$ acts by isometries on $X$). And in that case we can take $U_x$ to be a \emph{Dirichlet domain}\index{Dirichlet domain} $$U_x = \mathcal{D}_G(x) := \{y\in X\ |\ \forall\,g\notin G_x, d(y,x)<d(y,x\cdot g)\}.$$ Such a set is $G_x$-invariant because, if $y\in \mathcal{D}_G(x)$ and $h\in G_x$, the $G$-invariance of $d$ gives, for all $g\notin G_x$, $$d(y\cdot h, x\cdot g) = d(y, x\cdot gh^{-1}) > d(y,x) = d(y,x\cdot h^{-1}) = d(y\cdot h,x).$$ The set $\mathcal{D}_G(x)$ is also  disjoint from its translates by elements of $G$ that are not in $G_x$ because, if $g\notin G_x$, then we cannot have both $y\in\mathcal{D}_G(x)$ and $(y\cdot g)\in \mathcal{D}_G(x)$, as this would imply that $$d(y,x) < d(y,x\cdot g^{-1}) = d(y\cdot g,x) < d(y\cdot g, x \cdot g) = d(y,x).$$ So there only remains to prove that $\mathcal{D}_G(x)$ is open. For the sake of completeness, we sketch here the arguments to be found for instance in \cite[Theorem~9.4.2, pp.227--228]{Beardon}. Recall that $x$ is fixed throughout. For all $g\notin G_x$, set $$H_g(x):=\{y\in X\ |\ d(y,x)<d(y,x\cdot g)\}\quad \text{and}\quad L_g(x) := \{y\in X\ |\ d(y,x) = d(y,x\cdot g)\}.$$ In particular, $H_g(x)$ is open, $L_g(x)$ is closed and $\mathcal{D}_G(x) = \cap_{g\notin G_x} H_g(x)$. Since $G$ is a discrete subgroup of a Lie group, it is countable. And since $G$ acts properly on $X$, a compact set $C\subset X$ can only intersect a finite number of $L_g(x)$. Indeed, this is a direct consequence of the fact that, if $G=\{g_0,g_1,\,\ldots\}$, then $$\forall\,n\ |\ g_n\notin G_x, d\big(x,L_{g_n}(x)\big) = \textstyle{\frac{1}{2}} d\big(x,x\cdot g_n\big) \underset{n\to+\infty}{\lra} +\infty.$$ Now let us fix $y\in \mathcal{D}_G(x) = \cap_{g\notin G_x} H_g(x)$ and consider the compact disk $D[y;\epsilon]$ centred at $y$. As we have just seen, the set $$\{g\notin G_x\ |\ D[y;\epsilon]\cap L_g(x)\not=\emptyset\} =: \{g_1,\,\ldots\,,g_r\}$$ is finite. Since $y\in H_{g}(x)$, we have that, for all $i\in\{1,\,\ldots\,,r\}$, there exists an integer $n_i\geq 1$ such that the open disk $D(y,\frac{\epsilon}{2^{n_i}})$ has an empty intersection with $L_{g_i}(x)$, and is therefore contained in $H_{g_i}(x)$. So, we can find an open neighbourhood $U_y := \cap_{i=1}^r D(y_i,\frac{\epsilon}{2^{n_i}})$ such that, for all $g\notin G_x$, $U_y\subset H_g(x)$, i.e.\ $U_y\subset \mathcal{D}_G(x)$, proving that $\mathcal{D}_G(x)$ is indeed open.
\end{rem}

\begin{rem}\label{isometric_action} Note that, in the situation of Theorem \ref{proper_actions_of_discrete_groups_by_isometries}, a $G$-invariant distance exists on $X=K\backslash H$ because $H$ possesses a left-invariant Riemannian metric. Indeed we have assumed that the Lie algebra of $H$ admits an \emph{$H$-invariant} positive definite scalar product (this holds for instance if $H$ is a real form of a \emph{reductive} complex Lie group, such as $H=\SL(2;\R)$ in Corollary \ref{modular_orbi_curve}). Note that the complex analytic structure of $\fh$ indeed comes from its description as the homogeneous space $\mathbf{SO}(2)\backslash \SL(2;\R)$ because of the exceptional isomorphism between $\SL(2;\R)$ and the Hermitian real form $\mathbf{SU}(1,1)$ of $\SL(2;\C)$.
\end{rem}

\begin{rem}
As explained after the proof of Corollary 3.5.11 in \cite[p.157]{Thurston_1997}, \emph{if $X$ is a space of the type considered in Theorem \ref{proper_actions_of_discrete_groups_by_isometries}}, Condition (P) in Remark \ref{properly_discont_actions} admits equivalent characterisations, all of them given in \cite[Definition~3.5.1.(ii)-(v), pp.153--154]{Thurston_1997}. For instance the fact that $G$ is a discrete subgroup of $H$, or the fact that the action is \emph{wandering}, in the following sense:
\begin{itemize}
\item[(W)] For all $x\in X$, there is an open neighbourhood $U_x$ of $x$ such that the set $\{g\in G\ |\ U_x\cdot g \cap U_x \neq \emptyset\}$ is finite.
\end{itemize} If $X$ is locally compact, then (P) $\Rightarrow$ (W). And if $X$ is a metric space such that $G$ acts on $X$ \emph{by isometries}, then (W) $\Rightarrow$ (P). We shall not need this notion in what follows.
\end{rem}

\section{Moduli spaces}\label{moduli_spaces_section}

\subsection{Coarse moduli spaces and fine moduli spaces}

It is sometimes desirable to approximate an analytic stack by a space, in an appropriately optimal sense that we now define.

\begin{definition}\label{CMS}\index{moduli space!coarse moduli space}
Let $\X$ be an analytic stack. A \emph{coarse moduli space} for $\X$ is a pair $(M,\pi)$ where $M\in\An$ and $\pi:\X\lra M$ is a morphism of stacks satisfying the following universal property:
$$\text{For\ all\ morphism\ of\ stacks}\ \phi:\X\lra N\ \text{with}\ N\in\An, \exists!\, \overline{\phi}:M\lra N\ |\ \overline{\phi}\circ\pi=\phi.$$
\end{definition}

Note that the morphism $\overline{\phi}:M\lra N$ is a morphism of stacks between two manifolds, so by the Yoneda lemma \ref{Yoneda} it is induced by a uniquely defined morphism of analytic manifolds. The condition that $(M,\pi)$ be a coarse moduli space for $\X$ is usually summed up by the commutativity of the following diagram.

$$
\begin{tikzcd}
\X \ar[r, "\phi"] \ar[d, "\pi"] & N\\
M \ar[ru, "\exists!\,\overline{\phi}"', dashed] & 
\end{tikzcd}
$$

Sometimes, an extra condition is added, namely the fact that the functor $\pi_\bullet:\X(\bullet)\lra M(\bullet)\simeq M$ be surjective, or in other words, that there be a surjective map from the set $\X(\bullet)/\negmedspace\sim$ of isomorphism classes of objects parameterised by the stack $\X$ to the underlying set of the manifold $M$. Note however that asking for this map to be bijective (as opposed to just surjective) might be too restrictive (insofar as, in some examples, the universal property and the surjectivity condition are satisfied, but not the bijectivity condition). Typically, this happens when $\X=[X/G]$ and the action of $G$ on $X$ has non-closed orbits (see Example \ref{CMS_for_quotient_stacks_examples}). To understand this, we first need to study the existence of coarse moduli spaces for quotient stacks. The key observation is the following result.

\begin{prop}\label{univ_ppty_of_quot_stacks}
Let $G$ be a complex Lie group acting to the right on an analytic manifold $X$. Then, the canonical morphism $p:X\lra[X/G]$ defines a bijection $\phi\lmt \phi\circ p$ between morphisms of stacks $\phi$ from $[X/G]$ to a manifold $N\in\An$ and $G$-invariant morphisms $\psi:X\lra N$ in $\An$:
\begin{equation*}
\left\{
\text{morphisms\ of\ stacks}\ 
\phi:[X/G]\lra N
\right\}
\quad \simeq \quad
\left\{
G\text{-invariant\ morphisms}\ \psi:X\lra N\ \text{in}\ \An
\right\}
\end{equation*} More precisely, given a $G$-invariant morphism $\psi:X\lra N$ in $\An$, there is a unique morphism of stacks $\phi:[X/G]\lra N$ such that $\phi\circ p = \psi$.
\end{prop}

Proposition \ref{univ_ppty_of_quot_stacks} says that the quotient stack $[X/G]$ satisfies the following universal property with respect to $G$-invariant morphisms $ \psi: X \lra N$.

$$
\begin{tikzcd}
X \ar[r, "\psi"] \ar[d, "p"] & N\\
\left[X/G\right] \ar[ru, "\exists!\,\phi"', dashed] & 
\end{tikzcd}
$$

\begin{proof}[Proof of Proposition \ref{univ_ppty_of_quot_stacks}]
First note that, as in Example \ref{open_substacks_in_quotient_stack}, we have a $2$-Cartesian diagram 
\begin{equation}\label{pullback_diag_for_triv_bdle}
\begin{tikzcd}
X \times_{\left[X/G\right]} X \simeq X\times G \ar[r,"\text{action}"] \ar[d, "pr_1"] & X\ar[d, "p"] \\
X \ar[r, "p"'] & \left[X/G\right]
\end{tikzcd}
\end{equation} whose commutativity can be taken as a definition that the morphism of stacks $p:X\lra[X/G]$ is $G$-invariant. The commutativity property also implies that, if $\phi:[X/G]\lra N$ is a morphism of stacks whose target stack is a manifold $N\in\An$, then the morphism $\psi:=\phi\circ p:X\lra N$ is $G$-invariant as a morphism in $\An$. Now assume that a $G$-invariant morphism $\psi:X\lra N$ has been given in $\An$. We will show that, conversely, there is a unique morphism of stacks $\phi:[X/G]\lra N$ such that $\phi\circ p=\psi$, .i.e.\ that the universal property depicted in Diagram \eqref{univ_ppty_of_quotient_stacks_diagram} holds for the quotient stack $[X/G]$.

\begin{equation}\label{univ_ppty_of_quotient_stacks_diagram}
\begin{tikzcd}
X\times G \ar[r,"\text{action}"] \ar[d, "pr_1"] & X\ar[d, "p"] \ar[rdd, bend left=20, "\psi"]& \\
X \ar[r, "p"'] \ar[rrd, bend right=20, "\psi"] & \left[X/G\right] \ar[dr, dashed, "\exists!\,\phi"'] & \\
& & N
\end{tikzcd}
\end{equation}

\noindent Recall that to define $\phi$ is to define $\phi_S:[X/G](S)\lra \ul{N}(S)$ for all $S\in\An$. But an object in $[X/G](S)$ is a pair $(P,u)$ consisting of a principal $G$-bundle $P$ on $S$ and a $G$-equivariant map $u:P\lra X$. So we have a $G$-invariant morphism $\psi\circ u:P\lra N$ in $\An$. Since $P$ is a principal $G$-bundle on $S$, the morphism $\psi\circ u$ factors into a unique morphism $f:S\simeq P/G\lra N$ (by the universal property of the quotient $P/G$). 
$$
\begin{tikzcd}
P \ar[d] \ar[r, "u"] & X \ar[r,"\psi"] & N\\
S \ar[rru, dashed, "\exists\,!f"'] & &
\end{tikzcd}
$$
The morphism $\phi_S:[X/G](S)\lra \ul{N}(S)$ is then defined by $\phi_S(P,u) := f$.
\end{proof}

As an immediate consequence of Proposition \ref{univ_ppty_of_quot_stacks}, we get the following special case of Definition \ref{CMS}: \emph{for quotient stacks, the notion of coarse moduli space coincides with the classical notion of categorical quotient}. In particular, \textit{depending on the action, they may or may not exist}.

\begin{cor}\label{CMS_and_cat_quot}\index{quotient!categorical quotient}
Let $G$ be a complex Lie group acting to the right on a complex analytic manifold $X$. The quotient stack $[X/G]$ admits a coarse moduli space if and only if the action of $G$ on $X$ admits a \emph{categorical quotient}, i.e.\ if there exists a pair $(M,\rho)$ where $M\in\An$ and $\rho:X\lra M$ is a morphism of analytic manifolds such that:
\begin{enumerate}
\item The map $\rho$ is $G$-invariant: for all $g\in G$ and all $x\in X$, $\rho(x\cdot g) = \pi(x)$.
\item For all $G$-invariant morphism of analytic manifolds $\psi: X \lra N$, there exists a unique morphism of analytic manifolds $\overline{\psi}:M\lra N$ such that the following diagram is commutative:
\begin{equation}\label{univ_ppty_quot_space}
\begin{tikzcd}
X \ar[r, "\psi"] \ar[d, "\rho"] & N\\
M \ar[ru, "\exists!\,\overline{\psi}"', dashed] & 
\end{tikzcd}
\end{equation}
\end{enumerate}
\end{cor}

A particularly good case is when $M=X/G$, i.e.\ when the orbit space $X/G$ is a categorical quotient. This happens when the orbit space satisfies the conditions of the following definition.

\begin{definition}\label{geom_quot_def}\index{quotient!geometric quotient}
Let $G$ be a complex Lie group acting to the right on an analytic manifold $X$ and denote by $ \rho : X \lra X/G $ the canonical projection. If the orbit space $X/G$, endowed with the quotient topology, admits a structure of complex analytic manifold with respect to which the holomorphic functions on an open subset $ U \subset X/G $ correspond, via pullback by $\rho$, to $ G $-invariant holomorphic functions on the open subset $\rho^{-1}(U) \subset X $, then $X/G$ is called a \textit{geometric quotient} for the $ G $-action on $ X $.
\end{definition}

It can be checked that a geometric quotient is a categorical quotient: it necessarily satisfies the universal property \eqref{univ_ppty_quot_space}. An example of a situation in which a geometric quotient exists is the case when $G$ acts freely and properly on $ X $ (see Example \ref{proper_free_action}). Note that, since the points of $X/G$ are closed and the canonical projection $\rho:X\lra X/G$ is continuous, \textit{a geometric quotient can only exist if the $G$-orbits in $X$ are closed}, and if there are non-closed orbits, then the points of a categorical quotient (if it exists) cannot be in bijection with the orbits of the action (see Example \ref{CMS_for_quotient_stacks_examples}).  There are cases in which the orbits are closed and the action is not free but a geometric quotient exists: this is in fact what happens for the action of the modular group $\SL(2;\Z)$ on the upper half-plane $ \fh $ (see also Example \ref{roots_of_unity_acting_by_rotation}).

\begin{example}\label{proper_free_action}
If a complex Lie group $ G $ acts freely and properly on a complex analytic manifold $ X $, then the orbit space $ X / G $ admits a structure of complex analytic manifold that turns the canonical projection $ \rho : X \lra X/G $ into a holomorphic principal $ G $-bundle. In particular, the map $ \rho $ has local sections and holomorphic functions on $ X / G $ do indeed correspond to $ G $-invariant holomorphic functions on $ X $, so $ X / G$ is indeed a geometric quotient when $ G $ acts freely and properly on $ X $. As a matter of fact, we will see that more is true in that case: when the action is proper and free, the orbit space $ X / G $ is a \textit{fine} moduli space for the quotient stack $[X/G]$ (see Definition \ref{fine_mod_spaces_def} and Proposition \ref{fine_mod_space_for_quot_stacks}).
\end{example}

\begin{example}\label{CMS_for_quotient_stacks_examples}
Let $ G = \C^* $ act on $ X = \C^2 $ via $ t \cdot ( z_1, z_2 )  := ( t \cdot z_1, t^{-1} \cdot z_2 ) $. The closed orbits of this action are the origin $\{(0,0)\}$, which is a fixed point of the action, and the hyperbolas of equation $ z_1 z_2 = \lambda $, for all $ \lambda\in\C^*$. The origin is contained in the closure of the only two non-closed orbits, which are the real and imaginary axes minus the origin. In particular, a $G$-invariant morphism $\psi: X \lra N$ will take the same value on these three orbits, which will therefore be identified in the categorical quotient. This gives us the intuition that the categorical quotient in this example should be the set of closed orbits of the $G$-action, which is in bijection with $\C$ via the $ G $-invariant function $\rho: \C^2 \lra \C$ sending $( z_1, z_2 )$ to $z_1z_2$. By construction, this set has strictly less elements than the whole set of orbits of the $ G $-action. Sometimes, it can even have drastically less elements: if instead of the above action, we let $ G = \C^* $ act on $ X = \C^2 $ via $ t \cdot ( z_1, z_2 )  = ( t \cdot z_1, t \cdot z_2 ) $, then the only closed orbit is $\{(0,0)\}$, which gives us the intuition that the categorical quotient will be a point in this case.
\end{example}

In Section \ref{invariant_theory_section}, we will give sufficient conditions for the existence of categorical and geometric quotients for actions of \emph{reductive} complex Lie groups. As we shall see, in certain cases, a geometric quotient $X/G$ can exist even if the action of $ G $ on $ X $ is not free: a simple example of this is given in Example \ref{roots_of_unity_acting_by_rotation}. We end this section with the notion of fine moduli space, which is the case when the stack $\chi$ is not only approximated by a manifold $M$, but actually represented by it.

\begin{definition}\label{fine_mod_spaces_def}\index{moduli space!fine moduli space}
Let $\X$ be an analytic stack. A \emph{fine moduli space} for $\X$ is a complex analytic manifold $M$ such that $ \X \simeq M$. 
\end{definition}

It is readily checked that a fine moduli space is also a coarse moduli space for $ \X $. As an example, let us study the existence of fine moduli spaces for a quotient stack $\X = [X/G]$. 

\begin{prop}\label{fine_mod_space_for_quot_stacks}
Let $G$ be a complex analytic group acting to the right on a complex analytic manifold $X$. The quotient stack $[X/G]$ admits a fine moduli space if and only if $G$ acts freely and properly on $ G $. In that case, the canonical projection $p:X \lra X/G$ defines a holomorphic principal $G$-bundle.
\end{prop}

\begin{proof}
Let us first assume that $[X/G]$ admits a fine moduli space $M$. Then the canonical morphism $X \lra [X/G]\simeq M$ is a principal $G$-bundle, so $G$ acts freely and properly on $X$ and $M\simeq X/G$. Conversely, if $G$ acts freely and properly on $X$, let us consider the principal $G$-bundle $X \lra X/G$. For all $S\in\An$, an object in $[X/G](S)$ is a principal $G$-bundle $P\lra S$ equipped with a $G$-equivariant map $u:P\lra X$. Since $X/G$ is a manifold, the $G$-equivariant morphism $u$ induces a morphism of complex analytic manifolds $\hat{u}:P/G\simeq S\lra X/G\simeq M$. Conversely, such a map $f:S\lra M$ induces, by pullback of the principal $G$-bundle $X\lra X/G\simeq M$, a principal $G$-bundle $f^*X\lra S$, equipped by construction with a $G$-equivariant map $P\lra X$. These two maps set up natural isomorphisms $[X/G](S) \simeq M(S)$, hereby proving that $[X/G]$ is representable by $M$.
\end{proof}

\begin{example}\label{fine_mod_spaces_example}
The action of $\C^*$ on $\C^2\setminus\{(0,0)\} $ given by $t \cdot ( z_1, z_2 )  := ( t \cdot z_1, t \cdot z_2 ) $ is proper and free. So the quotient stack $[ \, (\C^2\setminus\{(0,0)\}) \, /  \, \C^*]$ admits $\CP^1$ as its fine moduli space.
\end{example}

We can reformulate Proposition \ref{fine_mod_space_for_quot_stacks} by saying that, for quotient stacks, a fine moduli space exists if and only if the action is free and admits a geometric quotient.

\subsection{GIT quotients}\label{invariant_theory_section}

To make sense out of Example \ref{CMS_for_quotient_stacks_examples} from a general point of view, we can use the Kempf-Ness formulation of Geometric Invariant Theory (\cite{Kempf,Kempf_Ness}). The work of Kempf and Ness applies to analytic actions of \emph{reductive} complex Lie groups $ G $ (i.e.\ groups that are isomorphic to the complexification of their maximal compact subgroup) and, at least when $ X $ can be embedded equivariantly and linearly onto a closed submanifold of some $ \C^N $ (meaning $G$-equivariantly with respect to a $ G $-action on $ \C^N$ that is induced by a linear representation $ G \lra \GL(N;\C)$), it enables us to say us two things:
\begin{enumerate}
\item If $ O \subset X $ is a $ G $-orbit, then its closure $ \overline{O} $ contains a unique closed orbit. As a consequence, the relation $ O_1 \sim O_2 $ if $ \overline{O_1} \cap \overline{O_2} \neq \emptyset $ is an equivalence relation on the set of $ G $-orbits (the previous result on the closure of an orbit being used to prove that this relation is transitive).
\item The set of closed $ G $-orbits in $ X $, denoted by $ X /\!\!/ G $, which by definition is the quotient of $ X / G $ by the previous equivalence relation, is a categorical quotient for the $ G $-action on $ X $. In particular, the holomorphic functions on $X/\!\!/G$ are the $G$-invariant holomorphic functions on $ X $. This quotient is called the \textit{GIT quotient} for the action of $ G $ on $ X $.
\end{enumerate}
There is in fact a more general situation, in which $ X $ can be embedded $ G$-equivariantly onto a locally closed submanifold of some $ \CP^{N-1}$ with respect to a \emph{linearisable} action of $G$, meaning that $ G $ is required to act on $ \CP^{N-1} $ via a a linear representation $ G \lra \GL(N;\C)$ and the canonical action of $ \GL(N;\C)$ on $\CP^{N-1}$. In that case, one has to restrict the action to the so-called \emph{semistable locus} $ X^{ss} \subset X $, which is a $G$-invariant open subset of $ X $ that can be defined as follows:
$$ X^{ss} = \{ x \in X \subset \CP^{N-1}\ | \ 0 \not\in \overline{G\cdot \tilde{x}}, \text{ where } \tilde{x}\in \C^N \text{ is a lift of } x\}. $$ If $X^{ss}\neq\emptyset$, then one can look at $ G $-orbits that are closed in $ X^{ss} $ and proceed as before to form the categorical quotient $ X^{ss} /\!\!/G $. The result, however, can depend strongly on the linearisation (see Example \ref{proj_space_as_GIT_quot}). In practice, it suffices to test semistability with respect to all $1$-parameter subgroups $\lambda: \C^* \lra G$, which is done via the Hilbert-Mumford criterion:
$$ X^{ss} = \big\{ x \in X \subset \CP^{N-1}\ | \ \forall\, \lambda: \C^* \lra G, \ \exists\, \mu \geq 0, \ \lim_{t \to 0} t^\mu\, \big( \lambda(t)\cdot \tilde{x} \big) \text{ exists in } \C^N \text{ and is} \neq 0 \big\} . $$

\begin{example}\label{proj_space_as_GIT_quot}
Let us equip $\C^2$ with the $\C^*$-action given by $t \cdot (z_1,z_2) = (tz_1,tz_2)$, and identify it with an open subset $ X \subset \CP^2$ via $(z_1,z_2) \lmt [1:z_1,z_2]$, then:
\begin{enumerate}
\item With respect to the linearisation $t\cdot[z_0:z_1:z_2] := [tz_0:tz_1:tz_2]$, it can be checked (via the Hilbert-Mumford criterion) that $X^{ss} = \emptyset$, so there is no categorical quotient in this case.
\item With respect to the so-called trivial linearisation $t\cdot[z_0:z_1:z_2] := [z_0:tz_1:tz_2]$, the semistable locus is $$X^{ss} =  \big\{ [ z_0 : z_1 : z_2] \in \CP^2 \ | \ z_0\neq 0 \big\} = X \simeq \C^2.$$ The only closed orbit is $\{(0,0)\}$ so the categorical quotient is a point in this case.
\item With respect to the linearisation $t\cdot[z_0:z_1:z_2] := [t^{-1}z_0:tz_1:tz_2]$, the semistable locus is 
$$ X^{ss} = \big\{ [ z_0 : z_1 : z_2] \in \CP^2 \ | \ z_0\neq 0 \text{ and } (z_1,z_2) \neq (0,0) \big\} \simeq \C^2 \setminus 
\big\{ (0,0) \big\} \, .$$ All $\C^*$-orbits are then closed in $X^{ss}$, so in this case the categorical quotient is 
$$ X^{ss} /\!\!/ \C^* \simeq  \big(\C^2\setminus\{(0,0)\}\big) \, /  \, \C^* \simeq \CP^1 .$$
\end{enumerate}
\end{example}

For a presentation of geometric invariant theory in a complex analytic setting (using the results of Kempf and Ness), we refer for instance to \cite{Woodward_CIRM}, the point being that, via that theory, linearisable actions of reductive complex Lie groups on quasi-projective analytic manifolds give rise to special categorical quotients called GIT quotients\index{quotient!GIT quotient}. More precisely, we have the following result.

\begin{prop}
Let $ G $ be a reductive complex Lie group acting linearly on a quasi-projective analytic manifold $ X \subset \CP^{N-1} $. If $ X^{ss} \neq \emptyset $, then the GIT quotient
$$ X^{ss} /\!\!/ G := \{ \text{closed } G\text{-orbits in } X^{ss} \}$$ 
is a coarse moduli space for the analytic stack $[X^{ss}/G]$.
\end{prop} 

\begin{proof}
Geometric invariant theory tells us that, as $ G $ is reductive, the space of closed semistable orbits $X^{ss} /\!\!/G$ carries a structure of complex analytic manifold that makes it a categorical quotient for the action of $ G $ on $ X^{ss}$. By Corollary \ref{CMS_and_cat_quot}, the GIT quotient $X^{ss} /\!\!/ G $ is therefore a coarse moduli space for the analytic stack $[X^{ss}/G]$. 
\end{proof}

\begin{rem}
When a reductive group $ G $ acts on a locally closed analytic submanifold $X \subset \C^N$, we can embed $ X $ equivariantly in $\CP^{N-1}$ with respect to the trivial linearization $ g \cdot [z_0: w] = [z_0 : g\cdot w] $, where $w\in\C^N$. Then $X^{ss} = X$, so in that case the quotient stack $[X/G]$ admits a coarse moduli space.
\end{rem}

\begin{example}\label{roots_of_unity_acting_by_rotation}
Let us retake the example of the orbifold $[\C/\mu_n(\C)]$, where the group $ \mu_n(\C) \simeq \Z/n\Z $ of $n$-th roots of unity acts on $\C$ via $ \zeta \cdot z = \zeta z $ (Examples \ref{covers_of_quotient_stack} and \ref{quotient_by_finite_and_discrete_group}). This action is proper but not free. Since $\mu_n(\C)$ is a finite group, it is reductive. So we get a GIT quotient whose orbits are all closed, and we note that such a GIT quotient is necessarily a geometric quotient in the sense of Definition \ref{geom_quot_def}, because by construction the holomorphic functions on a GIT quotient $X^{ss} /\!\!/ G$ correspond to $ G $-invariant functions on $ X^{ss} $. In the present example, we can identify this geometric quotient with $\C$ as a complex analytic manifold, via the holomorphic map $ j : \C/\mu_n(\C) \lra \C $ induced by the $\mu_n(\C)$-invariant map $z \lmt z^n$ (see Diagram \ref{CMS_roots_of_unity_acting_by_rotation}). The same occurs if we replace $ \C $ by an open disk centred at the origin.
\begin{equation}\label{CMS_roots_of_unity_acting_by_rotation}
\begin{tikzcd}
\C \ar[rd, "z\lmt z^n"] \ar[d, "\rho"'] & \\
\C/\mu_n(\C) \ar[r, "j"'] & \C
\end{tikzcd}
\end{equation}
\end{example}

To sum up, when $ G $ is a reductive complex analytic Lie group acting linearly on a quasi-projective complex analytic manifold $ X $ with non-empty semistable locus $X^{ss}$, GIT quotients $X^{ss} /\!\!/ G$ exist in $\An$ and satisfy the following properties:
$$ 
\text{GIT quotient} \Rightarrow \text{coarse moduli space/categorical quotient for the quotient stack} \left[X^{ss}/G\right] $$
and
$$ \text{GIT quotient + closed orbits} \Rightarrow \text{geometric quotient for the quotient stack} \left[X^{ss}/G\right]. 
$$
By definition, the points of the GIT quotient $X^{ss} /\!\!/ G$ are the closed $ G $-orbits in $X^{ss}$, so saying that all orbits are closed is exactly saying that $X^{ss} /\!\!/ G= X^{ss} / G$ (i.e.\ that the GIT quotient is an orbit space). Finally:
$$ 
\text{GIT quotient + closed orbits + free action} \Rightarrow \text{fine moduli space for the quotient stack } \left[X^{ss}/G\right] .
$$


\section{The moduli stack of elliptic curves}\label{stack_of_ell_curves}

\subsection{Families of elliptic curves}\label{def_of_stack_of_ell_curves}

Given a complex analytic manifold $S\in\An$, we define $\M_{1,1}(S)$ to be the category of \emph{analytic families} of elliptic curves parameterised by $S$, in the following sense.

\begin{definition}\label{family_ell_curves_over_S}\index{family!of elliptic curves}
A \textit{family of elliptic curves} parameterised by a complex analytic manifold $ S \in \An $ is a triple $(E,\pi,e)$ where: 
\begin{enumerate}
\item $E$ is a complex analytic manifold.
\item $\pi:E\lra S$ is a submersion in $\An$.
\item $e:S\lra E$ is a section of $\pi$.
\item For all $s\in S$, the pair $(E_s:=\pi^{-1}(\{s\}),e(s))$ is a complex elliptic curve, i.e.\ a compact connected Riemann surface $E_s$ of genus $1$, equipped with a marked point $e(s)\in E_s$. In particular, the map $\pi$ is surjective.
\end{enumerate} We define a \textit{morphism of families} from $(E_1,\pi_1,e_1)$ to $(E_2,\pi_2,e_2)$ to be an isomorphism $u:E_1\lra E_2$ in $\An$ such that $\pi_2\circ u = \pi_1$ and $u\circ e_1 = e_2$.
\begin{equation*}
\begin{tikzcd}
E_1 \ar[rr, "u"] \ar[dr, "\pi_1"] & & E_2 \ar[dl, "\pi_2"'] \\
& S \ar[ur, dashed, bend right=20, "e_2"'] \ar[lu, dashed, bend left=20, "e_1"] & 
\end{tikzcd}
\end{equation*}
\end{definition}

\begin{example}\label{Legendre_family_bis}
The Legendre family $E = \sqcup_{u\in \C\setminus\{0,1\}} \sC_u $ defined in Example \ref{Legendre_family} is a family of elliptic curves in the sense of Definition \ref{family_ell_curves_over_S}, parameterised by $S:=\C\setminus\{0,1\}$. Indeed, 
$$ E := \left\{ (u, [x:y:z]) \in S \times \mathbb{CP}^{2} \ | \ y^{2}z=x(x-z)(x-u z)\right\}$$ is a complex analytic submanifold of $S \times \mathbb{CP}^{2} $, mapping to $ S =  \C \setminus\{0,1\} $ via the projection $\pi$ to the first factor, which is a submersion. The fibre of $\pi: E \lra S$ above $u\in S$ is the genus $1$ Riemann surface $\sC_u$ constructed in Example \ref{Legendre_family}, and the points $[0:0:1]\in\sC_u$ define a global section $e: s \lmt (s, [0:0:1])$ of $ \pi$.
\end{example}

We shall often drop $\pi$ from the notation and simply write $(E,e)$ for a family of elliptic curves parameterised by $S$. Since we have restricted ourselves to isomorphisms between analytic families, the category $\M_{1,1}(S)$ is a groupoid, for all $ S \in \An $. To check that this defines a prestack $$ \M_{1,1}:\An^{\mathrm{op}} \lra \text{Groupoids} $$ in the sense of Definition \ref{def_prestack}, we must check that, for all morphism $f:T\lra S$ in $\An$, we can define a natural pullback functor $f^{\,*}:\M_{1,1}(S)\lra \M_{1,1}(T)$. To do this, we set, for all $(E,\pi,e)\in \M_{1,1}(S)$:
\begin{equation}\label{pullback_curve}
f^*E := \{(x,t)\in E\times T\ |\ \pi(x) = f(t)\}
\end{equation} Since $\pi:E\lra S$ is a holomorphic submersion, the topological space $f^*E$ admits a structure of complex analytic manifold, turning the canonical projection $f^*E\lra T$ into a submersion. Moreover, for all $t\in T$, the fibre $(f^*E)_t \simeq E_{f(t)}$ has a canonical structure of complex elliptic curve, with marked point $e(f(t))=(f^*e)(t)$. This proves that $(f^*E,f^*e)$ is an analytic family of elliptic curves parametrised by $T$ and that $f^*E$ if the fibre product of $E$ and $T$ over $S$ in $\An$. So $E\lmt f^*E$ is a functor with respect to $E$ and, if $g:U\lra T$ is a morphism in $\An$, we can identify canonically $g^*(f^*E)$ and $(f\circ g)^*E$ in $\M_{1,1}(U)$, which is what we need for Definition \ref{def_prestack}.



\begin{rem}
In the general definition of a family of complex analytic manifolds, the condition that the holomorphic submersion $\pi:E\lra S$ also be proper is often incorporated in the definition. In particular, with that definition, the fibres of $\pi$ are automatically compact.
\end{rem}

\begin{definition}\label{M_1_1_def}\index{moduli stack!of elliptic curves}\index{stack!moduli stack}
The prestack $\M_{1,1}$ defined, for all complex analytic manifold $S\in\An$, by $$\M_{1,1}(S) = \big\{\text{analytic families\ of\ elliptic\ curves\ parameterised\ by}\ S\big\}$$ is called the \emph{moduli stack} of elliptic curves.
\end{definition}

The term \emph{stack} used in Definition \ref{M_1_1_def} is justified by the following result.

\begin{prop}\label{prestack_M11_is_a_stack}
The prestack $\M_{1,1}$ is a stack.
\end{prop}

\begin{proof}
We have to check that the conditions of Definition \ref{def_stack} are satisfied by the prestack $\M_{1,1}$. This means showing that analytic families of elliptic curves parameterised by $S\in\An$, and morphisms between them, can be constructed by gluing. Explicitly, we want to show that, given an analytic manifold $ S $ and an open covering $ ( S_i )_{ i \in I } $ of $ S $, the following two properties hold:
\begin{enumerate}
\item If $ E, E' \lra S $ are families of elliptic curves parameterised by $ S $ and $ \phi_i : E|_{ S_i } \lra E'|_{ S_i } $ are morphisms between the pullbacks of these two families to $ S_i $, satisfying the compatibility conditions  $ \phi_i|_{ S_i \cap S_j } = \phi_j|_{ S_i \cap S_j } $ over $ S_i \cap S_j $, then there exists a unique morphism $ \phi : E \lra E ' $ in $ \M_{1,1}( S ) $ such that, for all $ i $, $ \phi|_{ S_i } = \phi_i $.
\item If, for all $ i $ in $ I $, we are given a family of elliptic curves $ \pi_i : E_i \lra S_i $, as well as isomorphisms $$\phi_{ij}: E_j|_{S_i\cap S_j} \overset{\simeq}{\lra} E_i|_{S_i\cap S_j}$$ satisfying the cocycle conditions $\phi_{ii} = \id_{E_i}$ and $\phi_{ij}\circ\phi_{jk} = \phi_{ik}$ over $ S_i \cap S_j \cap S_k$, then there exists a family of elliptic curves $ E \lra S $ together with isomorphisms $\phi_i:E|_{S_i}\lra E_i$ over $ S_i $, such that, for all $i,j$ in $ I $, one has $ \phi_i \circ \phi_j^{-1} = \phi_{ij} $ over $ S_i \cap S_j $.
\end{enumerate}
The first property holds because we can simply define $ \phi( x ) := \phi_i( x ) $ if $ x \in E|_{ S_i } $ and this gives a well-defined map that is a morphism of families of elliptic curves from $ E $ to $ E' $, with the property that $ \phi|_{ S_i } = \phi_i $. And the second property holds because the complex analytic manifold 
$$ E := \sqcup_{i \in I} E_i \, \big/ \sim $$ where $( i, x) \sim (j, y)$ if $ x = \phi_{ij}( y ) $ satisfies the required properties by construction: it is equipped with a holomorphic submersion
\begin{equation}\label{hol_submersion_constructed_by_gluing}
\pi : E \lra \left(\sqcup_{i \in I} S_i \, \big/ \sim\right) = S 
\end{equation} whose fibres are elliptic curves because the gluing procedure guarantees that $ E $ comes equipped with isomorphisms between the open subsets $ E|_{S_i} \subset E $ and the family $ E_i $.
\end{proof}

\begin{rem}
Note that if we had asked, in Definition \ref{prestack_M11_is_a_stack}, that the holomorphic submersion $ \pi: E \lra S $ also be proper, Property (2) in the proof of Proposition \ref{prestack_M11_is_a_stack} would be more difficult to show: we would also need to prove that the map \eqref{hol_submersion_constructed_by_gluing}, constructed by gluing, is proper, which is in fact the case when the morphisms $\pi_i: E_i \lra S_i$, from which $ \pi : E \lra S $ is constructed, are also assumed to be proper. 
\end{rem}

We ultimately want to argue that the reason why the stack $ \M_{1,1} $ is called a \emph{moduli} stack is because of the existence of a \emph{universal family}\index{family!universal family} on it. This will require an explanation, but the idea is that there exists, because of the geometric content behind the definition of $ \M_{1,1} $, a family of elliptic curves $ \U_{1,1} \lra \M_{1,1}$ with the following property:
\begin{equation}\label{univ_ppty_of_univ_family}
\begin{matrix}
\text{For all } S \in \An \text{ and all } E \in \M_{1,1}(S) \text{, there exists} \text{ a unique}\\ \text{morphism } f: S \lra \M_{1,1}  \text{ such that } f^* \U_{1,1} \simeq E .
\end{matrix}
\end{equation}
Note that, for Property \eqref{univ_ppty_of_univ_family} to make sense, we must provide a definition of a family of elliptic curves parameterised by a stack instead of a manifold. Once this is done, we can get an intuition of how we can define a universal family.

\begin{definition}\label{family_ell_curves_over_stack}
Let $\N$ be a stack over $\An$. A \textit{family of elliptic curves}\index{family!of elliptic curves} over $\N$ is a morphism of stacks $\chi: \E \lra \N $ such that:
\begin{enumerate}
\item The morphism $\chi: \E \lra \N$ is representable.
\item For all $S \in \An$ and all morphism $f : S \lra \N$, the pullback $$ f^*\E \simeq S \times_\N \E \in \An $$ is a family of elliptic curves over $ S $, in the sense of Definition \ref{family_ell_curves_over_S}.
\end{enumerate}
\end{definition}

In Definition \ref{family_ell_curves_over_stack}, the role of Condition (1) is to ensure that Condition (2) makes sense. Indeed, Condition (1) guarantees that $f^*\E$ is a manifold for all  morphism $f : S \lra \N$ with $S \in \An$.

\medskip

In view of Definition \ref{family_ell_curves_over_stack}, we can extend $\M_{1,1}$ to the category of stacks by defining groupoids 
\begin{equation}\label{families_param_by_a_stack} 
\widehat{\M_{1,1}}(\N) := \{ \E \lra \N : \ \text{family of elliptic curves on} \ \N \}\,.
\end{equation} 
Note that a family of elliptic curves $\E\lra \N$ induces a morphism of stacks $F: \N \lra \M_{1,1}$, defined, for all $ S \in \An $, by taking $(u:S \lra \N)$ to $u^*\E \in \M_{1,1}(S)$. For this to be compatible with the Yoneda embedding $\M_{1,1} \lmt \Mor_{\mathcal{S}t}(-,\M_{1,1})$, it is necessary and sufficient to assume that the the identity morphism $\id_{\M_{1,1}}$ comes, in the same fashion, from a family $\U_{1,1} \lra \M_{1,1}$. If this is the case, then by the same argument as in Example \ref{towards_Yoneda}, the family $\U_{1,1}$ must indeed satisfy the universal property \eqref{univ_ppty_of_univ_family}. So in order to define a universal family of elliptic curves, it suffices to be able to interpret $\id_{\M_{1,1}}$ as a family of elliptic curves over the stack $\M_{1,1}$.

\begin{rem}
The preceding discussion suggests that all stacks can be abstractly thought of as moduli stacks: if $\M$ is a stack on $\An$, we can consider the $2$-functor $\Mor_{\mathcal{S}t}(-,\M) : \mathcal{S}t/\An \lra \text{Groupoids}$ and define an extension of $\M:\An \lra \text{Groupoids}$ to $\widehat{\M}: \mathcal{S}t/\An \lra \text{Groupoids}$ by setting $\widehat{\M}(\N) := \Mor_{\mathcal{S}t}(-,\M)$ for all $\N\in\mathcal{S}t/\An$. The role of the universal family $\N\in\widehat{\M}(\M)$ is then played by the morphism $\id_{\M}$. But ultimately we want to have a more concrete (and geometrically interpretable) description of universal families, which will depend on the precise stack $\M$ being studied.
\end{rem}

When $\M=\M_{1,1}$, we will describe the universal family of elliptic curves in Corollary \ref{description_of_univ_family_of_ell_curves}. For the moment, we content ourselves with two elementary but fundamental examples of universal families.

\begin{example}\label{universal_bundle}
Let $G$ be a complex analytic group and let $BG:= [\bullet/G]$ be the classifying stack of principal $G$-bundles. By Example \ref{towards_Yoneda}, for all manifold $S \in \An$, the groupoid $BG(S)$ is the   category of principal $G$-bundles on $X$, and by Example \ref{quot_stacks_are_analytic}, the stack $BG$ is an analytic stack. The universal family is then the morphism $EG:=\bullet \lra BG = [ \bullet / G]$. Indeed, it is a representable morphism by Example \ref{representability_of_the_can_morphism_from_X_to_X_mod_G} and, by Example \ref{towards_atlas_for_BG}, it satisfies the universal property \eqref{univ_ppty_of_univ_family}. This means that, for all $S\in \An$ and all principal $G$-bundle $P$ on $S$, there exists a unique morphism of stacks $f:S\lra BG$ such that $P\simeq f^* EG$. Because of this universal property, the principal $G$-bundle $EG\lra BG$ is called the \emph{universal bundle}\index{bundle!universal bundle} over the classifying stack $BG$. Note that $EG$ is indeed a principal $G$-bundle\index{bundle!principal bundle on a stack} over the stack $BG$, where this notion is to be taken in a sense similar to the covering stacks of Definition \ref{stacky_cover_def}: a representable morphism $EG\lra BG$ whose pullback to a manifold $S\in\An$ along a morphism of stacks $S \lra BG$ is a principal $G$-bundle in $\An$. Also, comments similar to Proposition \ref{check_opennness_on_atlas} and Remark \ref{rep_morph_for_analytic_stack} apply: since $BG$ is an analytic stack, to check that the representable morphism $EG \lra BG$ is a principal $G$-bundle, it suffices to check it with respect to the atlas $EG \lra BG$, over which the pullback $EG$ is just $EG \times G$, the trivial principal $G$-bundle over $EG=\bullet$, as seen by taking $X=\bullet$ in Diagram \eqref{pullback_diag_for_triv_bdle}.
\end{example}

\begin{rem}\label{rep_of_univ_fmly_does_not_imply_rep}
Note that $BG$ is not a representable stack as soon a $G$ is non-trivial, but that the universal family $EG=\bullet$ is (representable by) a manifold.
\end{rem}

\begin{example}
As a generalisation of Example \ref{universal_bundle}, consider the principal $G$-bundle $X \lra [X/G]$, where $G$ is a complex analytic group acting on a complex analytic manifold $X$. To say that this is the universal family over the stack $[X/G]$ defined in Example \ref{quotient_prestack} is to say that, for all $S\in\An$ and all object $(P,u)\in [X/G](S)$, there exists a unique morphism of stacks $f: S \lra [X/G]$ such that $P \simeq f^*X$ and that the $G$-equivariant morphism $f^*X \lra X$ coincides with $u: P \lra X$. In other words, it means finding an $f: S \lra [X/G]$ such that the diagram 
$$
\begin{tikzcd}
P \ar[r,"u"'] \ar[d] & X\ar[d] \\
S \ar[r, "f"'] & \left[X/G\right]
\end{tikzcd}
$$ is a pullback diagram. This is possible because the principal $G$-bundle $P$ is locally trivial over $S$ and we have already noted that, if $P = X\times G$ and $u: X\times G \lra X$ is the action map, then  Diagram \eqref{pullback_diag_for_triv_bdle} is a pullback diagram.
\end{example}

\subsection{Families of framed elliptic curves}

We now wish to study the representability of the moduli stack $\M_{1,1}$. As we shall see, this stack can be seen not to be representable for rather general and abstract reasons, which have to do with the fact that the objects that it parameterises, namely elliptic curves, have non-trivial automorphisms (Corollary \ref{autom_of_elliptic_curves_and_genus_one_Riemann_surfaces_again}) and this prevents the moduli stack of elliptic curves from being representable (Proposition \ref{not_rep_if_non_triv_autom}). Since one can get rid of automorphisms by adding a framing to elliptic curves (Theorem \ref{Teich_space_of_ell_curves}), a natural route to study $\M_{1,1}$ is to introduce a \emph{moduli stack of framed elliptic curves}\index{moduli stack!of framed elliptic curves} $\M_{1,1}^{\,\mathrm{fr}}$ and relate it to $\M_{1,1}$. We will ultimately show that $\M_{1,1}^{\,\mathrm{fr}}$, in contrast to $\M_{1,1}$, is representable (by the Teichmüller space $\fh$) and this will imply that the moduli stack $\M_{1,1}$ is isomorphic to the quotient stack $[\fh/\SL(2;\Z)]$. This last step is what shows that $\M_{1,1}$ is an orbifold in the sense of Definition \ref{orbifold_def} (in particular, an analytic stack with atlas $\fh$) and that it admits the modular curve $\fh/\SL(2;\Z)$ as a coarse moduli space. But before we can get to all this, let us make a couple of remarks. 

\begin{rem}
Let us assume that there exists a family of elliptic curves $\U_{1,1} \lra \M_{1,1}$ (in the sense of Definition \ref{family_ell_curves_over_stack}) such that the associated morphism $\M_{1,1} \lra \M_{1,1}$ (defined, for all $S\in\An$, by sending a morphism of stacks $f: S \lra \M_{1,1}$ to the family of elliptic curves $f^*\U_{1,1} \in \M_{1,1}(S)$) is the identity morphism. Then, as we have noted, the stack $\U_{1,1}$ is a universal family in the sense that it satisfies the universal property \eqref{univ_ppty_of_univ_family}. As suggested by Example \ref{rep_of_univ_fmly_does_not_imply_rep}, the representability of the universal family need not imply the representability of the moduli stack. In the reverse direction, however, the representability of $\M_{1,1}$ (by a manifold $M_{1,1}$, say) would imply that the universal family $\U_{1,1}$ is representable (by the fibre product $M_{1,1}\times_{\M_{1,1}} \U_{1,1}$, which is indeed a manifold because the morphism of stacks $\U_{1,1} \lra \M_{1,1}$ is by assumption representable). 
\end{rem}

\begin{rem}
More generally, if we can prove that $\M_{1,1}$ is an analytic stack, then the universal family $\U_{1,1}$ (which we have not yet constructed), is automatically an analytic stack. Indeed, since $\U_{1,1}\lra\M_{1,1}$ is representable, the fibre product of $\U_{1,1}$ with an atlas $X \lra \M$ gives a manifold $X\times_{\M_{1,1}} \U_{1,1}$, which is an atlas for $\U_{1,1}$ because it is straightforward to check that all morphisms in the pullback diagram
$$
\begin{tikzcd}
X\times_{\M_{1,1}} \U_{1,1} \ar[r] \ar[d] & \U_{1,1}\ar[d] \\
X \ar[r] & \M_{1,1}
\end{tikzcd}
$$
are representable surjective submersions. 
\end{rem}

Independently of these remarks, we can prove the following result.

\begin{prop}\label{not_rep_if_non_triv_autom}
The moduli stack of elliptic curves $\M_{1,1}$ is not representable by a manifold.
\end{prop}

\begin{proof}
Let us assume that the stack $\M_{1,1}$ is representable by a manifold $M_{1,1}\in\An$, i.e.\ that $\M_{1,1} \simeq \Mor_{\An}(-,M_{1,1})$. The contradiction will come from a special property of representable functors such as $\Mor_{\An}(-,M_{1,1})$, namely that if $f: X \lra Y$ is a covering map in $\An$, then the induced map 
\begin{equation}\label{pullback_by_cover}
f^*: \Mor_{\An}(Y,M_{1,1}) \lra \Mor_{\An}(X,M_{1,1})
\end{equation} is injective. In fact, it suffices to know that this property holds for Galois covers $ Y = X / G$ where $G$ is a finite group acting by automorphisms on $Y$. In this case, the image of the pullback map \eqref{pullback_by_cover} is the set $\mathrm{Fix}_G\, \Mor_{\An}(X,M_{1,1})$ of $G$-invariant complex analytic maps from $X$ to $M_{1,1}$ and since such a map induces a well-defined complex analytic map from $Y$ to $M_{1,1}$, the pullback map $f^*$ is injective. 

\smallskip

Thus, to prove that $\M_{1,1}$ is not representable, it suffices to find a Galois covering map $\pi: X \lra X/G$ and two non-isomorphic families of elliptic curves $E_1, E_2$ on $X/G$ such that $\pi^*E_1 \simeq \pi^* E_2$ as families of elliptic curves over $X$. Indeed, this will contradict the injectivity of $\pi^*: \M_{1,1}(Y) \lra \M_{1,1}(X)$ which has to hold if we assume that $\M_{1,1} \simeq \Mor_{\An}(-,M_{1,1})$, i.e.\ that $\M_{1,1}$ is representable. And this is when the fact that there are elliptic curves with non-trivial automorphisms comes in handy (to construct the families $E_1, E_2$).

\smallskip

Let $\sC$ be the elliptic curve $\sC := \C / (\Z1 \oplus \Z i)$ and set $X := \sC$. We consider the action of the group $G := \Z/2\Z$ induced on $E := X\times \sC$ by the involution 
\begin{equation}\label{concrete_invol}
\big([x], [z]\big) \lmt \big( [x+\textstyle\frac{1}{2}], [-z] \big)
\end{equation} and the quotient $E_1 := E / G$. Since $G$ acts freely on $X$, the quotient $Y:= X/G$ is a Riemann surface and the canonical projection $\pi:X \lra Y$ is a two-to-one cover. Note that $Y$ has genus $1$ but does not carry a preferred marked point because $G$ does not act on $X$ by automorphisms of elliptic curves. Moreover, there is an induced complex analytic map $E_1 \lra Y$, which is locally trivial with fibres isomorphic to $\sC$. Since the action of $G$ on the fibres of $E$ is an action \textit{by automorphisms of elliptic curves}, the map $E_1 \lra Y$ is in fact a family of elliptic curves. The point is that $E_1$ is not trivial but pulls back to the trivial family $E = X \times \sC$. Since the product family like $E_2 := Y \times \sC$ also pulls back to $E$, we see that the map $\pi^*: \M_{1,1}(Y) \lra \M_{1,1}(X)$ is not injective. 

\smallskip

Note that, to see that $E_1$ is not trivial, it suffices to show that it does not carry a non-vanishing holomorphic volume form (which a product of two Riemann surfaces of genus $1$ would), because otherwise the pullback of such a volume form to $E = X\times \sC$ would be $G$-invariant and equal to $c dx \wedge dz $ for some constant $c\neq 0$, contradicting the fact that $dx \wedge dz $ is $G$-anti-invariant with respect to the involution \eqref{concrete_invol}.
\end{proof}

\begin{rem}
The technique used in the proof of Proposition \ref{not_rep_if_non_triv_autom} can be adapted to show the non-representability of various moduli stacks. It rests on the fact that a representable stack $\M \simeq \Mor_{\An}(-,M)$ is a sheaf in various Grothendieck topologies on $\An$, in particular the étale topology. So, to prove that a stack $\M$ is not representable, it suffices to find an étale covering $\pi: Y\lra X$ such that $\pi^*:\M(X) \lra \M(Y)$ is not  injective. Indeed, this injectivity is part of the conditions of being a sheaf: the restriction map (or pullback) to an open covering (or more generally an étale cover) is an injective map (if two sections of a sheaf coincide on every open of the cover, they are equal). Then, in order to show a lack of injectivity, the usual technique (as in the proof of Proposition \ref{not_rep_if_non_triv_autom}) consists in finding a non-trivial family whose pullback to a finite cover is trivial. The construction of such a family (called an \textit{isotrivial family}\index{family!isotrivial family}) is then typically obtained by taking the quotient of a trivial family by a finite group acting \textit{by automorphisms on the fibre} and freely on the base.
\end{rem}

We now want to introduce families of \emph{framed} elliptic curves. Recall that, in Definition \ref{framed_ell_curve}, we defined a framed elliptic curve to be a triple $(\sC,e,(\alpha,\beta))$ consisting of an elliptic curve $(\sC,e)$ equipped with a direct basis $(\alpha,\beta)$ of $H_1(\sC;\Z)$, i.e.\ a basis satisfying the positivity condition \eqref{positivity_cond}, and this was convenient in order to classify framed elliptic curves and prove that they have no non-trivial automorphisms (Theorem \ref{Teich_space_of_ell_curves}). While we can generalise this point of view to families of elliptic curves, it is not well-adapted to the more general moduli theory of higher genus curves (see also Remark \ref{a_rem_about_higher_genus}). So, before we proceed, we give an equivalent of a framed elliptic curve, based on the following result.

\begin{lemma}\label{framing_alt_def}
Let $(\sC_0,e_0,(\alpha_0,\beta_0))$ and $(\sC_1,e_1,(\alpha_1,\beta_1))$ be framed elliptic curves. Then there exists an orientation-preserving homeomorphism $f:(\sC_0,e_0) \lra (\sC_1,e_1)$ such that $f_*(\alpha_0) = \alpha_1$ and $f_*(\beta_0) = \beta_1$.
\end{lemma}

Note that the condition of being an orientation-preserving homeomorphism from $\sC_0$ to $\sC_1$ makes sense because complex analytic elliptic curves, like all Riemann surfaces, are canonically \emph{oriented} topological surfaces. In the compact connected case, this means that a group isomorphism $H_2(\sC_i;\Z)\simeq \Z$ has been chosen, and for a homeomorphism $f:\sC_0 \lra \sC_1$ to be orientation-preserving means that the induced group isomorphism $f_*: H_2(\sC_0;\Z) \lra H_2(\sC_1;\Z)$ commutes to these identifications. Equivalently, the orientation of $\sC_i$ gives a notion of direct basis on $H_1(\sC_i;\Z)$, and the induced group isomorphism $f_*: H_1(\sC_0;\Z) \lra H_1(\sC_1;\Z)$ sends a direct basis to a direct basis.

\begin{proof}[Proof of Lemma \ref{framing_alt_def}]
By Theorem \ref{Teich_space_of_ell_curves}, we can assume that $(\sC_0,e_0,(\alpha_0,\beta_0))$ and $(\sC_1,e_1,(\alpha_1,\beta_1))$ are framed elliptic curves of the form $(\C/\Lambda(\tau_i),0,(1,\tau))$ for some $\tau_0,\tau_1\in\C$ such that $\mathrm{Im}(\tau_i)>0$. Then let us consider the map $\tilde{f}: \C \lra \C$ defined by 
$$
z = x+iy  \lmt  \frac{(\tau_1-\overline{\tau_0}) z - (\tau_1-\tau_0) \overline{z}}{ \tau_0 - \overline{\tau_0}} = x + \frac{\mathrm{Re}(\tau_1 - \tau_0)}{\mathrm{Im}(\tau_0)} y + i \frac{\mathrm{Im}(\tau_1)}{\mathrm{Im}(\tau_0)} y 
$$ which is the orientation-preserving invertible $\R$-linear map from $\C\simeq \R^2$ to itself defined by the matrix 
\begin{equation}\label{deformation_of_lattices}
\begin{pmatrix}
1 & \mathrm{Re}(\tau_1 - \tau_0) / \mathrm{Im}(\tau_0) \\
0 & \mathrm{Im}(\tau_1) / \mathrm{Im}(\tau_0)
\end{pmatrix}
\end{equation}
and is the unique $\R$-linear map from $\R^2$ to itself sending the basis $(1,\tau_0)$ to the basis $(1,\tau_1)$. By construction, we have $\tilde{f}(m+n\tau_0) = m + n \tau_1$, so $\tilde{f}$ induces an orientation-preserving homeomorphism $f: \C/\Lambda(\tau_0) \lra  \C/\Lambda(\tau_1)$, sending $0$ to $0$ while also taking the framing $(1,\tau_0)$ to the framing $(1,\tau_1)$.
\end{proof}

Note that, by the Cauchy-Riemann equations, the map $\tilde{f}:\C\lra \C$ is holomorphic if and only if the matrix \eqref{deformation_of_lattices} is a similitude matrix, which happens if and only if $\tau_0=\tau_1$. This is consistent with the fact that framed elliptic curves have no non-trivial holomorphic automorphisms. More importantly for us, Lemma \ref{framing_alt_def} implies that there is a bijection between framings of an elliptic curve $(\sC_0,e_0)$ and homotopy classes of orientation-preserving self-homeomorphisms $f\in \mathrm{Homeo}^+(\sC_0,e_0)$, where the latter notation means homeomorphisms $f:\sC_0 \lra \sC_0$ that preserve the canonical orientation of $\sC_0$ and that satisfy $f(e_0) = e_0$. Indeed, the action of the \emph{modular group}\index{modular!modular group}
\begin{equation}\label{modular_group_bis}
\mathrm{Mod}(\sC_0,e_0) := \mathrm{Homeo}^+(\sC_0,e_0) \, / \, \text{homotopy}
\end{equation} on the set of direct bases of $H_1(\sC;\Z)\simeq\Z^2$ is transitive by Lemma \ref{framing_alt_def} and free because, as $\sC$ has genus $1$, its universal cover is $\C$ so an orientation-preserving self-homeomorphism of $(\sC,e)$ lifts to a $\pi_1(\sC,e)$-equivariant orientation-preserving self-homeomorphism of $(\C,0)$ and, as $\C$ is contractible, two such transformations are $\pi_1(\sC,e)$-equivariantly homotopic precisely when they induce the same group automorphism of $H_1(\sC;\Z)$, which can be seen as a lattice in $\C$. Thus, in the situation of Lemma \ref{framing_alt_def}, two orientation-preserving homeomorphisms $f,g:(\sC_0,e_0) \lra (\sC_1,e_1)$ that both send $(\alpha_0,\beta_0)$ to $(\alpha_1,\beta_1)$ will satisfy $g^{-1}\circ f \sim \id_{\sC_0}$, so $g$ will indeed be homotopic to $f$ in $\mathrm{Homeo}^+(\sC_0,e_0)$. Note that the fact that the modular group of a complex analytic elliptic curve $(\sC_0,e_0)$ acts freely and transitively on the set of direct bases of $H_1(\sC_0;\Z)\simeq\Z^2$ also proves that 
\begin{equation}\label{modular_group_of_a_torus_with_marked_point}
\mathrm{Mod}(\sC_0,e_0) \simeq \SL(2;\Z)
\end{equation}
so Definition \ref{modular_group_bis} is consistent with the terminology introduced in Section \ref{modular_curve_section}. We can now redefine framed elliptic curves as follows (compare Definition \ref{framed_ell_curve}).

\begin{definition}\label{framed_ell_curve_alt_def}\index{elliptic curve!framed elliptic curve}\index{framing}
Let us fix a complex analytic elliptic curve $(\sC_0,e_0)$. Then, given a complex elliptic curve $(\sC,e)$, we define a \textit{framing} $[f]$ of $(\sC,e)$ to be a homotopy class of orientation-preserving homeomorphism, i.e.\
$$
[f] \in \mathrm{Homeo}^+\big((\sC_0,e_0),(\sC,e)\big) \, / \, \text{homotopy}.
$$
The set of all framings of $(\sC,e)$ will be denoted by $\mathrm{Fr}_{(\sC_0,e_0)}(\sC,e)$. It is acted upon freely and transitively by the modular group $\mathrm{Mod}(\sC_0,e_0)$, via the right action defined, for all 
$[g]\in \mathrm{Mod}(\sC_0,e_0)$ by 
\begin{equation}\label{mod_gp_action}
[f]\cdot[g] := [f\circ g].
\end{equation}
\end{definition}

Thanks to Definition \ref{framed_ell_curve_alt_def}, we will define a framing of a family of elliptic curves $(E,\pi:E\lra S,e)$ in the sense of Definition \ref{family_ell_curves_over_S} essentially as a framing of the elliptic curve $(E_s,e(s))$ for all $s\in S$, i.e.\ as the data, for all $s\in S$, of a framing $\sigma(s) \in \mathrm{Fr}_{(\sC_0,e_0)}(E_s,e(s))$, where $(\sC_0,e_0)$ is a fixed elliptic curve. In order for $\sigma(s)$ to depend continuously on $s$, it suffices to construct a fibre bundle $\mathrm{Fr}\,(E,e) \lra S$ whose fibre over $s\in S$ is $\mathrm{Fr}_{(\sC_0,e_0)}(E_s,e(s))$ and to define $\sigma$ as a continuous section of that fibre bundle. We will follow Grothendieck's construction in \cite{Groth_constr_axiom_Teich}.

\medskip

So let us start with a family of elliptic curves $(E,\pi:E\lra S,e)$. By Definition \ref{family_ell_curves_over_S}, this means in particular that $\pi:E\lra S$ is a surjective holomorphic submersion with compact connected fibres over the manifold $S$. By the so-called monotone-light factorisation (see for instance \cite[p.102]{Duda_Whyburn}), such a map is necessarily proper. But then, the fact that $\pi$ is a submersion implies that, for all $z\in E$, there exists an open neighbourhood $U$ of $x:=\pi(z)$ in $S$, an open neighbourhood $V\subset\pi^{-1}(U)$ of $z$ in $E$, an open subset $\Omega\subset\pi^{-1}(\{x\})$ and a biholomorphic map $V \simeq U \times \Omega$ over $U$ (\cite[Proposition 6.2.3 p.240]{Hubbard}). Such a map is called \emph{simple} in \cite{Groth_constr_axiom_Teich}, where the base of a family $E\lra S$ is a general complex analytic space, not necessarily a manifold. Note that, unless $V=\pi^{-1}(U)$ and $\Omega= \pi^{-1}(\{x\})$, this is strictly weaker than $\pi$ being a locally trivial holomorphic vector bundle. However, a simple holomorphic map $\pi:E\lra S$ is a locally trivial \emph{topological} fibre bundle. When $S$ is a manifold, this is Ehresmann's theorem, since a simple map is submersive, and for the case when $S$ is an analytic space, which we will not require, we refer to \cite[Lemma 2.1, p.3]{Groth_constr_axiom_Teich}. The point here is that a family of elliptic curves is in particular a locally trivial continuous map $\pi:E \lra S$, whose fibres are oriented compact connected topological surfaces of genus $1$ equipped with a marked point, meaning that for all $ s \in S$, the set $\mathrm{Homeo}^+((\sC_0,e_0),(E_s,e(s)))$ is non-empty. In particular, we may view $\pi:E\lra S$ as a topogical fibre bundle equipped with a continous section $e:S\lra E$, with typical fibre $(\sC_0,e_0)$ and structure group $\mathrm{Homeo}^+(\sC_0,e_0)$. Then, by taking homotopy classes of the transition functions of $E$, we can associate to it the locally trivial topological fibre bundle with discrete fibres $p:\mathrm{Fr}\,(E,e) \lra S$ defined by
\begin{equation}\label{frame_bundle}
\mathrm{Fr}\,(E,e) :=  \bigsqcup_{s\in S} \mathrm{Fr}_{(\sC_0,e_0)}\big(E_s,e(s)\big)
\end{equation}
where 
$$
\mathrm{Fr}_{(\sC_0,e_0)}(E_s,e(s)) = \mathrm{Homeo}^+\big((\sC_0,e_0),(E_s,e(s))\big) \, / \, \text{homotopy} 
$$ so $\mathrm{Fr}\,(E,e)\lra S$ is actually a principal covering space of $S$, with structure group the discrete group $\mathrm{Mod}(\sC_0,e_0)$ introduced in \eqref{modular_group_bis}. We will call $\mathrm{Fr}\,(E,e)$ the \emph{frame bundle}\index{frame bundle} of the family $(E,\pi,e)$. Its construction only depends on the homeomorphism class of the fixed complex elliptic curve $(\sC_0,e_0)$, which is that of a torus $S^1\times S^1$, equipped with a marked point. This construction is functorial, in the sense that a morphism of families of elliptic curves $u:(E_1,e_1)\lra (E_2,e_2)$ (Definition \ref{family_ell_curves_over_S}) gives rise to a morphism of principal bundles
$$\mathrm{Fr}\,(u): \mathrm{Fr}\,(E_1,e_1) \lra \mathrm{Fr}\,(E_2,e_2)$$
induced fibrewise (for all $s\in S$) by the map
$$ 
\begin{array}{rcl}
\mathrm{Homeo}^+\Big( \big(\sC_0,e_0\big), \big(E_1|_{\{s\}},e_1(s)\big) \Big) & \lra & \mathrm{Homeo}^+\Big( \big(\sC_0,e_0\big), \big(E_2|_{\{s\}},e_2(s)\big) \Big) \\
f(s) & \lmt & u_s \circ f(s)
\end{array}
$$ where $u_s : E_1|_{\{s\}} \lra  E_2|_{\{s\}} $ is the isomorphism of elliptic curves induced by the isomorphism $u:(E_1,e_1) \lra (E_2,e_2)$ (recall that, by Definition \ref{family_ell_curves_over_S}, all morphisms of families are isomorphisms).

\begin{definition}\label{framed_family}\index{family!of framed elliptic curves}
Let us fix a complex analytic elliptic curve $(\sC_0,e_0)$. A \emph{family of framed elliptic curves} is a quadruple $(E,\pi,e,\sigma)$ where $(E,\pi:E\lra S,\sigma)$ is a family of elliptic curves parameterised by $S\in\An$ in the sense of Definition \ref{family_ell_curves_over_S} and $\sigma:S \lra \mathrm{Fr}\,(E,e)$ is a continuous section of the frame bundle $\mathrm{Fr}\,(E,e)$ of the family $(E,\pi,e)$, meaning that, for all $s\in S$, we can think of $\sigma(s)$ as (the homotopy class of) a homeomorphism $(\sC_0,e_0) \lra (E_s,e(s))$. A morphism between families of framed elliptic curves $$(E_1,\pi_1,e_1,\sigma_1) \lra (E_2,\pi_2,e_2,\sigma_2) $$ is a morphism of families of elliptic curves $u:(E_1,e_1)\lra (E_2,e_2)$ such that the induced morphism of principal covering spaces
$$\mathrm{Fr}\,(u): \mathrm{Fr}\,(E_1,e_1) \lra \mathrm{Fr}\,(E_2,e_2)$$
satisfies $\mathrm{Fr}\,(u) \circ \sigma_1 = \sigma_2$.
\begin{equation*}
\begin{tikzcd}
\mathrm{Fr}\,(E_1,e_1) \ar[rr, "\mathrm{Fr}\,(u)"] \ar[dr] & &  \mathrm{Fr}\,(E_2,e_2) \ar[dl] \\
& S \ar[ur, dashed, bend right=20, "\sigma_2"'] \ar[lu, dashed, bend left=20, "\sigma_1"] & 
\end{tikzcd}
\end{equation*}
\end{definition}

In particular, all morphisms of framed elliptic curves are isomorphisms, by definition. We shall often drop $\pi$ from the notation and simply write $(E,e,\sigma)$ for a family of framed elliptic curves parameterised by $S$.

\begin{rem}
Note that passing to homotopy classes in the definition of $\mathrm{Fr}\,(E,e)$ gives us a locally trivial fibration, with the topological group $\mathrm{Homeo}^+(\sC_0,e_0)$ replaced by the \emph{discrete} group $\mathrm{Mod}(\sC_0,e_0)$ as a fibre. In particular, we can endow $\mathrm{Fr}\,(E,e)$ with a canonical structure of complex analytic manifold: the one with respect to which the local homeomorphism $p: \mathrm{Fr}\,(E,e) \lra S$ is holomorphic. It then becomes a local biholomorphism ($=$ an étale analytic map) and the continuous sections of $\mathrm{Fr}\,(E,e)\lra S$ are exactly the holomorphic sections with respect to that holomorphic structure.
\end{rem}

If $f: T \lra S$ is a morphism in $\An$ and $(E,e,\sigma)$ is a family of framed elliptic curves on $S$, then the pullback family 
\begin{equation}
f^*E := \{(x,t)\in E\times T\ |\ \pi(x) = f(t)\}
\end{equation}
defined in \eqref{pullback_curve} is canonically framed by the section $f^*\sigma$ of the frame bundle $f^*\mathrm{Fr}\,(E,e) = \mathrm{Fr}\,(f^*E,f^*e)$. The next result is then proved similarly to Proposition \ref{prestack_M11_is_a_stack}.

\begin{prop}
The functor
$$
\M^{\,\mathrm{fr}}_{1,1}:
\begin{array}{ccl}
\An^{\text{op}} & \lra & \mathrm{Groupoids} \\
S & \lmt & \big\{\text{analytic families\ of\ framed\ elliptic\ curves\ parameterised\ by}\ S\big\} 
\end{array}
$$ 
is a stack, called the \emph{moduli stack of framed elliptic curves}\index{moduli stack!of framed elliptic curves}\index{stack!moduli stack}.
\end{prop}

It is also clear that there is a morphism of stacks from the moduli stack of framed elliptic curves to the moduli stack of elliptic curves, defined, for families parameterised by $S\in\An$, by forgetting the frame,
$$F:
\begin{array}{rcl} 
\M^{\,\mathrm{fr}}_{1,1}(S) & \lra & \M_{1,1}(S) \\
(E,e,\sigma) & \lmt & (E,e)
\end{array}
$$ since this is compatible with taking pullbacks.

\begin{rem}\label{a_rem_about_higher_genus}
We could also have defined the objects in $\M^{\,\mathrm{fr}}_{1,1}(S)$ as triples $(E,e,\sigma')$ where $(E,e)$ is a family of elliptic curves parameterised by $S\in \An$ and $\sigma'$ is a section of the bundle $\mathrm{Fr}\,(E,e)\times_{\mathrm{Mod}(\sC_0,e_0)} H^1(\sC_0,e_0)$, associated to the principal $\mathrm{Mod}(\sC_0,e_0)$-bundle $\mathrm{Fr}\,(E,e)$ via the action of the modular group on $H^1(\sC_0,e_0)$, such that, for all $s\in S$, the element $\sigma'(s)=(\alpha(s),\beta(s))$ is a direct basis of $H^1(\sC_0,e_0)$. The upshot is that this is the direct generalisation of Definition \ref{framed_ell_curve} to families of elliptic curves. The fact that one can use one bundle or the other is because, in genus $1$, the action of $\mathrm{Mod}(\sC_0,e_0)$ on $H^1(\sC_0,e_0)$ is \emph{faithful}, which is no longer the case in higher genus. Definition \ref{framed_family}, in contrast, generalises immediately to higher genus (where it is no longer necessary to provide the curves with a base point).
\end{rem}

\subsection{Orbifold structure and coarse moduli space}

Contrary to the moduli stack $\M_{1,1}$, which is not representable by Proposition \ref{not_rep_if_non_triv_autom}, the moduli stack $\M^{\,\mathrm{fr}}_{1,1}$ is representable. To understand why, recall that, for all $S\in\An$,
$$
\M^{\,\mathrm{fr}}_{1,1}(S) = \big\{ (E,e,\sigma) : \text{families of framed elliptic curves on } S\big\}$$ so, for all $s\in S$, the framed elliptic curves $(E_s, e(s), \sigma(s))$ should be isomorphic, by Theorem \ref{Teich_space_of_ell_curves}, to the framed elliptic curve $(\C/\Lambda(\tau_s),0,(1,\tau_s))$, for a unique $\tau_s \in \fh$. This defines a map 
\begin{equation}\label{period_map}
f_\sigma:
\begin{array}{rcl}
S & \lra & \fh \\
s & \lmt & \tau_s
\end{array}
\end{equation}
corresponding to the family $(E,e,\sigma)$, which is in fact a morphism of complex analytic manifolds. By naturality with respect to $S$, the correspondence $(E,e,\sigma) \lmt f_\sigma$ in \eqref{period_map} induces a morphism of stacks 
\begin{equation}\label{representability_by_fh}
\M^{\,\mathrm{fr}}_{1,1} \lra \Mor_{\An}(-,\fh)
\end{equation}
which will be seen to be an isomorphism if we can define a universal family of framed elliptic curves 
$$
\U^{\,\mathrm{fr}}_{1,1} \lra \M^{\,\mathrm{fr}}_{1,1}.
$$
Indeed, if for all $\tau\in \fh$, the framed elliptic curve $(\U^{\,\mathrm{fr}}_{1,1})|_{\{s\}}$ is isomorphic to $(\C/\Lambda(\tau_s),0,(1,\tau_s))$, then by construction of the map $f_\sigma$ introduced in \eqref{period_map}, we have $f_\sigma^*\U^{\,\mathrm{fr}}_{1,1} \simeq (E,e,\sigma)$ as families of framed elliptic curves, and $f_\sigma$ is the unique map with that property. So $\U^{\,\mathrm{fr}}_{1,1}$ is indeed a universal family of framed elliptic curves and the morphism of stacks defined, for all $S\in\An$, by
$$
\begin{array}{rcl}
\Mor_{\An}(S,\fh) & \lra & \M^{\,\mathrm{fr}}_{1,1}(S) \\
(f : S \lra \fh) & \lmt & f^* \U^{\,\mathrm{fr}}_{1,1}
\end{array}
$$ is an inverse to the morphism of stacks $(E,e,\sigma) \lmt f_\sigma$ introduced in \eqref{representability_by_fh}. To check that the previous considerations are correct and that $\M^{\,\mathrm{fr}}_{1,1}$ is indeed representable by $\fh$, we still need to:
\begin{enumerate}
\item Prove that the map $f_\sigma: S \lra \fh$ constructed in \eqref{period_map} is indeed holomorphic.
\item Construct a universal family of framed elliptic curves over $\fh$, meaning a family $\U^{\,\mathrm{fr}}_{1,1} \lra \fh$ such that, for all $\tau\in \fh$, the fibre of $\U^{\,\mathrm{fr}}_{1,1}$ at $\tau$ is isomorphic to the framed elliptic curve $(\C/\Lambda(\tau),0,(1,\tau))$.
\end{enumerate}

For the first part, recall that, because of Definition \ref{framed_family}, the definition of the moduli stack of framed elliptic curves requires fixing an arbitrary complex elliptic curve $(\sC_0,e_0)$ and that it is in fact only the homeomorphism type of that complex elliptic curve that matters. But the latter is uniquely determined: it is the homeomorphism type of the topological torus $S^1\times S^1$, with an arbitrary base point, for instance $(1,1)$. The point is that this topological torus comes equipped with a canonical direct homology basis $(\alpha_0,\beta_0)$ for $H_1(S^1\times S^1;\Z)$, and that all other bases can be brought onto that one by an orientation-preserving self-homeomorphism of $S^1\times S^1$. Alternately, we can also choose $(\sC_0,e_0)$ to be the complex torus corresponding to the lattice $\Lambda(i):=\Z 1\oplus \Z i \subset \C$, which comes equipped with the direct basis $(1,i)$. The point is that we can fix not only $(\sC_0,e_0)$ but also a direct homology basis $(\alpha_0,\beta_0)$ for $H_1(\sC_0;\Z)$ and, given a family of framed elliptic curves $(E,\pi:E\lra S,e,\sigma)$, use that direct basis  $(\alpha_0,\beta_0)$ to define the map $\phi_\sigma$ explicitly as follows:
\begin{equation}\label{period_map_explicit}
\phi_\sigma:
\begin{array}{rcl}
S & \lra & \fh \\
s & \lmt & \tau_s := \frac{\int_{\sigma(s)_*\beta_0} \omega_s}{\int_{\sigma(s)_*\alpha_0} \omega_s} = \frac{\int_{\beta_0} \sigma(s)^*\omega_s}{\int_{\alpha_0} \sigma(s)^*\omega_s}
\end{array}
\end{equation} 
where $\sigma(s): \C/\Lambda(i) \lra E(s)$ is the orientation-preserving homeomorphism determined by $\sigma$, and $\omega_s$ is a continuous family of holomorphic $1$-form on the elliptic curve $E(s)$. This last part means that $\omega$ is a section of the sheaf $\pi_*\Omega^1_{E/S}$, the push-forward of the sheaf of relative holomorphic differential forms of degree $1$ of $E$ over $S$. It makes sense to see these relative differentials as families of differential forms along the fibres because, as $\pi:E\lra S$ is a submersion by assumption, the canonical sequence of $O_E$-modules
$$ 
0 \lra \pi^* \Omega^1_S \lra \Omega^1_E \lra \Omega^1_{E/S} \lra 0
$$
is exact and locally split (see for instance \cite[Proposition~5, p.37]{BLR_Neron}). It then follows that the map $\phi_\sigma$ defined in \eqref{period_map_explicit} is indeed holomorphic.

\medskip

It therefore only remains to construct a universal family $\pi:\U^{\,\mathrm{fr}}_{1,1} \lra \fh$. Again, we do this explicitly. First we consider the (free and proper) action of the discrete Abelian group $\Z^2$ on the product $\fh\times\C$ defined as follows and which we write on the right. For all $(m,n)\in\Z^2$ and all $(\tau,z)\in\fh\times\C$, set 
\begin{equation}\label{action_of_Z2}
(\tau,z) \cdot (m,n):= (\tau, z + m \tau + n ).
\end{equation}
Next we take the quotient \emph{manifold}
\begin{equation}\label{univ_family_of_framed_ell_curves}
\U^{\,\mathrm{fr}}_{1,1} := (\fh\times\C) / \Z^2
\end{equation}
and consider the map $\pi:\U^{\,\mathrm{fr}}_{1,1} \lra \fh$ induced by the first projection $(\tau,z)\lmt\tau$, which is a well-defined  holomorphic submersion. The point is that, for all $\tau\in\fh$, the identity map of $\C$ induces an isomorphism of complex analytic manifolds
$$
\pi^{-1}\big(\{\tau\}\big) \simeq \C/\Lambda(\tau)
$$
where $\Lambda(\tau) = \Z1 \oplus \Z\tau$. Moreover, the family of complex tori $\U^{\,\mathrm{fr}}_{1,1}\lra \fh$ has a canonical section $e:\fh\lra \U^{\,\mathrm{fr}}_{1,1}$, induced by the section $\tau\lmt(\tau,0)$ of the first projection $\fh\times\C \lra \fh$. Finally, for all $\tau\in\fh$, we can let 
$$
\sigma(\tau): \C/\Lambda(i) \lra \C/\Lambda(\tau)
$$ be the orientation-preserving homeomorphism defined as in the proof of Lemma \ref{framing_alt_def}, by setting $\tau_0=i$ and $\tau_1=\tau$ in \eqref{deformation_of_lattices}. We have therefore defined a family of framed elliptic curves $(\U^{\,\mathrm{fr}}_{1,1},e,\sigma)$ in the sense of Definition \ref{framed_family}. And, by construction, we have, for all $\tau\in\fh$ an isomorphism of framed elliptic curves
$$
\Big( \U^{\,\mathrm{fr}}_{1,1}\big|_{\{\tau\}}, e(\tau), \sigma(\tau) \Big) \simeq \big( \C/\Lambda(\tau), 0, (1,\tau) \big)
$$
so the family $\U^{\,\mathrm{fr}}_{1,1}$ is indeed a universal family. We summarise the previous considerations as follows.

\begin{thm}
The moduli stack of framed elliptic curves $\M^{\,\mathrm{fr}}_{1,1}$ is representable by the Teichmüller space $\fh$.
\end{thm}

It remains to show that this representability theorem for $\M^{\,\mathrm{fr}}_{1,1}$ is compatible with the action on $\fh$ of the modular group $\mathrm{Mod}(\sC_0,e_0)$ which, as in \eqref{modular_group_of_a_torus_with_marked_point}, we identify with $\SL(2;\Z)$. First note that, for all $S\in\An$, all family of elliptic curves $(E,e)$ over $S$ and all $s\in S$, the modular group $\mathrm{Mod}(\sC_0,e_0)$ acts on the set of framings of the elliptic curve $(E_s,e(s))$. Indeed, according to Definition \ref{framed_ell_curve_alt_def}, a framing of $(E_s, e(s))$ is a homotopy class of orientation-preserving homeomorphism $f: \sC_0 \lra E_s$ such that $f(e_0) = e(s)$, and we have seen in \eqref{mod_gp_action} that there is a right action of $\mathrm{Mod}(\sC_0,e_0)$ on the set of such framings, defined by $[f]\cdot[g] = [f\circ g]$. This induces an action of $\mathrm{Mod}(\sC_0,e_0)$ on the set $\Mor_{\An}(S,\fh)$, under which the morphism $\phi_\sigma: S\lra \fh$ corresponding to the family $(E,e,\sigma)$ is sent to the morphism $\phi_{\sigma\cdot g}$ corresponding to the family $(E,e,\sigma\cdot [g])$, where, for all $g\in\mathrm{Homeo}^+(\sC_0,e_0)$, we denote by $\sigma\cdot [g]$ the following section of the frame bundle $\mathrm{Fr}(E,e)$ introduced in \eqref{frame_bundle}:
\begin{equation}\label{abstract_action_on_framings}
(\sigma\cdot [g])(s) = \sigma(s) \cdot [g]\,.
\end{equation}
There only remains to compute $\phi_{E\cdot[g]}$ more explicitly. Note that $\sigma\cdot[g]$ is a framing of the \emph{same} family of elliptic curves $(E,e)$. So, in view of Definition \ref{period_map_explicit} and the fact that $(f\circ g)_* = f_* \circ g_*$ in homology, we can use Theorem \ref{modular_gp_action} to conclude that 
\begin{equation}\label{equivariance_ppty}
\phi_{\sigma\cdot[g]}(s) = \phi_\sigma(s) \cdot [g],
\end{equation}
where $\sigma\cdot [g]$ is defined as in \eqref{abstract_action_on_framings} and the action on the right-hand side is the $\SL(2;\Z)$-action on $\fh$ defined in \eqref{action_mod_group}.

\medskip

We are now in a position to prove that there is an isomorphism of stacks $\M_{1,1} \simeq [\mathfrak{h}/\SL(2;\Z)]$, which by Corollary \ref{modular_orbi_curve} will show that the moduli stack of elliptic curves is an orbifold, and to identify the universal family $\U_{1,1}$ as an orbifold itself. In order to define a morphism of stacks $\M_{1,1} \lra [\mathfrak{h}/\SL(2;\Z)]$, note that, if $S\in\An$ and $(E,\pi:E\lra S,e)\in\M_{1,1}(S)$ is a family of elliptic curves parameterised by $S$, then there is a morphism of analytic varieties $u: \mathrm{Fr}(E,e) \lra \fh$ from the frame bundle of $(E,\pi:E\lra S,e)$ to the Teichmüller space $\fh$, which is defined as follows. Recall that an element $[f]\in\mathrm{Fr}(E,e)$ is, by Definition \ref{framed_ell_curve_alt_def} and Equation \eqref{frame_bundle}, the homotopy class of an orientation-preserving homeomorphism $f: (\sC_0,e_0) \lra (E_{\pi([f])}, e(\pi([f])))$, where $(\sC_0,e_0)$ is our fixed elliptic curve $(\C/\Lambda(i),0)$ and $p:\mathrm{Fr}(E,e) \lra S$ is the canonical projection of the frame bundle. So we can define an analytic map $u: \mathrm{Fr}(E,e) \lra \fh$ by sending $[f]\in\mathrm{Fr}(E,e)$ to the complex number $\tau_{p([f])}$ defined in \eqref{period_map_explicit}. For the same reason as in Equation \eqref{equivariance_ppty}, the map $u: \mathrm{Fr}(E,e) \lra \fh$ thus constructed is $\SL(2;\Z)$-equivariant with respect to the $\SL(2;\Z)$-action on $\mathrm{Fr}(E,e)$ defined in \eqref{mod_gp_action} and the usual $\SL(2;\Z)$-action on $\fh$. And since $\mathrm{Fr}(E,e)$ is a principal $\SL(2;\Z)$-bundle on $S$ and the contruction we have just given is compatible with pullbacks with respect to a morphism $T\lra S$, we have indeed a morphism of stacks $\M_{1,1} \lra [\fh/\SL(2;\Z)]$.

\begin{thm}\label{main_goal}
Let $S\in\An$. The natural correspondence
$$
\Phi_S:
\begin{array}{ccc}
\M_{1,1}(S) & \lra & [\fh/\SL(2;\Z)](S) \\
(E,\pi:E\lra S,e) & \lmt & 
\left(
\begin{tikzcd}
 \mathrm{Fr}(E,e) \ar[r, "u"'] \ar[d, "p"'] & \fh\\
S & 
\end{tikzcd}
\right)

\end{array}
$$
induces an isomorphism of stacks $\Phi: \M_{1,1} \overset{\simeq}{\lra} [\fh/\SL(2;\Z)]$ over $\An$. In particular, the moduli stack of complex elliptic curves is a complex analytic orbifold in the sense of Definition \ref{orbifold_def}, admitting the geometric quotient $\fh/\SL(2;\Z) \simeq \C$ as a coarse moduli space in the sense of Definition \ref{CMS}.
\end{thm}

\begin{proof}
We will define an inverse to $\Phi:  \M_{1,1} \lra [\fh/\SL(2;\Z)]$ in the following way. Given a pair $(P,u)$, consisting of a principal $\SL(2;\Z)$-bundle $P\lra S$ and an $\SL(2;\Z)$-equivariant map $u: P \lra \fh$, we consider the pullback family $\hat{E} := u^*\U_{1,1}^{\,\mathrm{fr}}$ over $P$. We claim that $\SL(2;\Z)$ acts on the universal family $\U_{1,1}^{\,\mathrm{fr}}$ and that this induces a free and proper action on the family $\hat{E}$, so we can consider the quotient family $E:= \hat{E}/\SL(2;\Z)$ over $S=P/\SL(2;\Z)$. The correspondence $(P,u)\lmt E$ then defines a natural transformation $\Psi_S: [\fh/\SL(2;\Z)](S) \lra \M_{1,1}(S)$, which provides an inverse to $\Phi$. In order to check some of the details, we spell out the action of $\SL(2;\Z)$ on the universal family $\U_{1,1}^{\,\mathrm{fr}}$.

\medskip

Recall from \eqref{univ_family_of_framed_ell_curves} that $\U^{\,\mathrm{fr}}_{1,1} = (\fh\times\C) / \Z^2$, where $\Z^2$ acts on $\fh$ via \eqref{action_of_Z2}. In fact, this $\Z^2$-action extends to an action of the semi-direct product $\Z^2 \rtimes \SL(2;\Z)$, where $\SL(2;\Z)$ acts to the left on $\Z^2$, via the standard representation. Explicitly, if we consider the semi-direct product $\Z^2 \rtimes \SL(2;\Z)$ with group law defined as follows
$$
\left( \big(m,n\big)\ ,\ \begin{pmatrix} a & c \\ b & d \end{pmatrix} \right)  
\left( \big(m',n'\big)\ ,\ \begin{pmatrix} a' & c' \\ b' & d' \end{pmatrix} \right)
= 
\left( \big(m, n\big) + \big(am'+cn', bm'+dn')\ ,\ \begin{pmatrix} a & c \\ b & d \end{pmatrix}  \begin{pmatrix} a' & c' \\ b' & d' \end{pmatrix} \right)\, ,
$$
then the formula
\begin{equation}\label{action_for_def_of_univ_family}
\big(\tau,z\big)\cdot \left( \big(m,n\big)\ ,\  \begin{pmatrix} a & c \\ b & d \end{pmatrix} \right) 
:=
\left( \frac{a\tau+b}{c\tau+d}\ ,\ \frac{z + m\tau +n}{c\tau+d} \right) 
\end{equation} defines a right action of $\hat{G} := \Z^2\rtimes \SL(2;\Z)$ on $\fh\times \C$. There is therefore an induced action of the quotient group $\SL(2;\Z) = \hat{G} / \Z^2 $ on the universal family $\U^{\,\mathrm{fr}}_{1,1} = (\fh\times\C) / \Z^2$. 

\medskip

Going back to the proof of Theorem \ref{main_goal}, we see that the action \eqref{action_for_def_of_univ_family} induces an action of $\SL(2;\Z)$ on $u^* \U^{\,\mathrm{fr}}_{1,1}$ which, by construction, makes the proper, surjective map $u^*\U^{\,\mathrm{fr}}_{1,1} \lra P$ equivariant. Since the $\SL(2;\Z)$-action on $P$ is free and proper, it is also free and proper on $u^*\U^{\,\mathrm{fr}}_{1,1}$, as proved for instance in \cite[Proposition~3.5.8, p.156]{Thurston_1997}.
\end{proof}

As a corollary, we obtain a description of the universal family of (unframed) elliptic curves $\U_{1,1} \lra \M_{1,1}$. Contrary to the universal family of framed elliptic curves $\U^{\,\mathrm{fr}}_{1,1}$, it is no longer (representable by) a manifold. We can nonetheless identify it explicitly as an analytic orbifold.

\begin{cor}\label{description_of_univ_family_of_ell_curves}
Let $\U_{1,1} \lra \M_{1,1}$ be the universal family of elliptic curves over the moduli stack $\M_{1,1}$. Then there is an isomorphism of stacks  $$\U_{1,1} \simeq \big[(\fh\times\C) \ \big/ \ \big(\Z^2\rtimes\SL(2;\Z)\big) \big]$$ where the semi-direct product $\Z^2\rtimes\SL(2;\Z)$ acts on $(\fh\times\C)$ as in \eqref{action_for_def_of_univ_family}. In particular, the universal family of elliptic curves is an analytic orbifold.
\end{cor}

The isomorphism $\M_{1,1} \simeq [\fh/\SL(2;\Z)]$ appears as Theorem~4.1 in \cite[p.12]{Groth_constr_axiom_Teich}, and the fact that $\fh/\SL(2;\Z)$ is a coarse moduli space for $\M_{1,1}$ is discussed in \cite[Theorem~5.2, p.15]{Groth_constr_axiom_Teich}. The link with $V$-manifolds in the sense of Satake is mentioned on p.16 of \cite{Groth_constr_axiom_Teich}.



\end{document}